\documentclass[dvips]{article}
\usepackage[utf8]{inputenc}
\usepackage[english]{babel}
\usepackage{amssymb}
\usepackage{amsmath}
\usepackage{amsthm}
\usepackage{dsfont}
\usepackage{mathrsfs}
\usepackage[pdftex]{graphicx}
\usepackage{caption}
\usepackage{enumerate}

\DeclareGraphicsExtensions{.eps}

\newtheorem{thm}{Theorem}[section]
\newtheorem{prop}[thm]{Proposition}
\newtheorem{lem}[thm]{Lemma}
\newtheorem{cor}[thm]{Corollary}
\newtheorem{con}[thm]{Conjecture}

\theoremstyle{definition}
\newtheorem{defn}[thm]{Definition}

\theoremstyle{remark}
\newtheorem{rmk}[thm]{Remark}

\newcommand{\Sv}{\Sigma_\varphi}
\newcommand{\Spsi}{\Sigma_\psi}
\newcommand{\Sd}{\Sigma_\delta}
\newcommand{\Sdd}{\Sigma_{\delta'}}
\newcommand{\gmk}{\gamma_{\varphi_-^k}}
\newcommand{\gpk}{\gamma_{\varphi_+^k}}

\date{\today}
\author{\textsc{Jullian} Yann}
\title{An algorithm to identify automorphisms which arise from self-induced interval exchange transformations}

\addcontentsline{toc}{section}{Introduction}

\begin{document}
\pagenumbering{arabic}
\begin{center}
\vbox{
	\vspace*{4.7cm}
	{\LARGE An algorithm to identify automorphisms which arise from self-induced interval exchange transformations}

	\vspace*{0.5cm}

	{\large Yann {\scshape Jullian}}\\

	\vspace*{0.8cm}

	\begin{abstract}
	We give an algorithm to determine if the dynamical system generated by a positive automorphism
	of the free group can also be generated by a self-induced interval exchange transformation.
	The algorithm effectively yields the interval exchange transformation in case of success.
	\end{abstract}

}
\end{center}

\setcounter{tocdepth}{2} 
\tableofcontents{}

\section*{Introduction}

This article deals with the problem of the geometric representation of the symbolic systems induced by free group automorphisms (see section \ref{subsec:attsub}).
It is now well known (\cite{GJLL}, \cite{CHL09}, \cite{Cou}, \cite{CH}) that these can be represented
by systems of partial isometries on $\mathds{R}$-trees (geodesic and $0$-hyperbolic metric spaces). In \cite{CH},
the authors study fully irreducible automorphisms, and explain the representation of their dynamics depends on an attribute of
the automorphism called the index (\cite{GJLL}, \cite{Jul}).
This index depends on the fixed subgroup and the infinite fixed points of the automorphism, and determines whether
the dynamics of the automorphism is represented by a system of partial isometries on a finite union of finite (with respect to
the number of points with degree at least $3$) trees (this includes intervals), a finite union of non-finite trees
(an example is given in \cite{Jul0}),
or a Cantor set whose convex hull is a finite (an example is given in \cite{BK}) or non-finite tree.

If we fix a basis $A_N$ of the free group $F_N$ with $N\ge 2$ generators, the free group is seen as the
set of reduced (two adjacent letters cannot be the inverse of each other) words with letters in $A_N$ or $A_N^{-1}$.
In this article, we consider $A_N$-positive primitive automorphisms;
an automorphism $\varphi$ is $A_N$-positive if for any $a\in A_N$, the letters of
the word $\varphi(a)$ are all in $A_N$, and it is primitive if there exists a positive integer $k$ such
that all the elements of $A_N$ appear in $\varphi^k(a)$ for any $a\in A_N$. This second property
ensures the minimality of the dynamic. While considering
only $A_N$-positive automorphisms is certainly a restriction, it should be noted that the primitivity
condition is weaker than the full irreducibility (a nice discussion on the subject can be found in \cite{ABHS}).

This article aims to answer the following question; how can we determine if an $A_N$-positive primitive
automorphism can be represented by a system of partial isometries on an interval? We will use the early results
on interval exchange transformations (\cite{Kea1}, \cite{Kea2}, \cite{Rau}, \cite{Vee}).
In particular, the Rauzy induction (\cite{Rau}) gives us a way to obtain the subshift
(the combinatorial representation) of an interval exchange transformation using sequences of
elementary automorphisms (called Dehn twists). We say that an interval exchange transformation
is self-induced when this sequence is periodic.
The main result of this paper is the following.

\begin{thm}
	There is an algorithm able to determine in a finite number of steps if an $A_N$-positive primitive
	automorphism generates the dynamic of a self-induced interval exchange transformation.

	In case of success, the interval exchange transformation is fully determined and a Dehn twists decomposition
	of the automorphism is given.
\end{thm}

The algorithm is described in great details. It is based on two main ideas. The first one is the study made in \cite{Jul}
of the index of an $A_N$-positive primitive automorphism (summarized in appendix \ref{appendix:thebase}) and the second one
is a combinatorial interpretation of the Rauzy induction.

\subsection*{Outline of the paper}

Section \ref{sec:autom} deals with basic notions regarding free group automorphisms.
We endow the free group $F_N$ ($N\ge 2$) with a basis $A_N$ and consider $F_N$ (resp. $\partial F_N$) to be the set of
finite (resp. infinite) reduced (two adjacent letters cannot be the inverse of each other) words with letters in $A_N\cup A_N^{-1}$
(where $A_N^{-1}$ is the set of inverse letters).
The shift map $S$ is defined on each point $(X, Y)$ of the set $\partial^2F_N = (\partial F_N\times \partial F_N) \setminus \Delta$
(where $\Delta$ is the diagonal) by $S(X, Y) = (Y_0^{-1}X, Y_0^{-1}Y)$, with $Y_0$ being the first letter of $Y$.
\begin{rmk}
	Readers who are more familiar with the definitions of symbolic dynamic may want to see a point $(X, Y)$ of $\partial^2 F_N$
	as the bi-infinite word $X^{-1}Y$. Additionaly, we refer to \cite{CHL1}, which gives a detailed comparison between
	standard concepts of group theory and symbolic dynamic.
\end{rmk}

We will primarily deal with positive automorphism; $\varphi$ is a positive automorphism if all the letters
of $\varphi(a)$ are in $A_N$ for any $a\in A_N$. These are also called invertible substitutions.
In section \ref{subsec:attsub}, we define the attracting subshift of a positive automorphism.
The attracting subshift $\Sv\subset \partial^2 F_N$ of $\varphi$ is simply the closure of the $S$-orbit of a certain point $(U, V)$ that is
fixed by $\partial^2 \varphi$ (the homeomorphism induced by $\varphi$ on $\partial^2 F_N$).
This notion is equivalent to the bi-infinite symbolic system induced by an invertible substitution (see \cite{Que}).

In section \ref{subsec:autsing} we define the singularities of an automorphism (see \cite{Jul}).
In most cases, a singularity of an automorphism $\varphi$ is a set $\Omega$ of points of the attracting subshift with a common first (resp. second) coordinate.
In terms of bi-infinite words, these correspond to the existence of special (i.e which can be extended in more than one way) left-infinite or right-infinite sequences.
This is a key part of the present paper as these singularities are a natural equivalent to the discontinuity points of
interval exchange transformations.\\

An interval exchange transformation (IET) $f$ on an interval $I$ is a bijective map characterized by a set
$\{I_a; a\in A_N\}$ of intervals partitioning $I$ and a translation vector attached to each element of the partition
(see \cite{Kea1}, \cite{Kea2}, \cite{Rau}, \cite{Vee}, \cite{Yoc}). Each interval
$I_a$ and the interval $I$ itself are conventionally closed on the left and open on the right.
In section \ref{sec:ciet}, we introduce the important notion of completed interval exchange transformations (CIETs).
If $f$ is as above, we define the CIET $\delta$ associated to $f$ as a system $\{\delta_a; a\in A_N\}$ of partial isometries
such that, for any $a\in A_N$, the domain of $\delta_a$ is the closure of $I_a$ and $\delta_a(x)=f(x)$ for any $x\in I_a$.
The important part of this definition is that some points now have two images (resp. pre-images) under the action of $\delta$.
These points are called the forward (resp. backward) $\delta$-singularities. The existence of such points implies that
some points of the interval will have more than one (bi-infinite) orbit under the action of $\delta$.
We associate a point of $\partial^2 F_N$, called the coding, to each orbit.
The set of coding of orbits is the $\delta$-subshift $\Sd$ (see section \ref{sec:ciet}).

In sections \ref{subsec:keane} and \ref{subsec:rauzyind}, we recall the definitions of Keane condition
(see \cite{Kea1}, \cite{Kea2}) and Rauzy induction (see \cite{Rau}) for IETs
and adapt these notions for CIETs. Let $f$ be an IET on an interval $I$; assume $\{I_a; a\in A_N\}$ is a set of intervals
partitioning $I$ and $f$ is a translation on each $I_a$. Define $J_0$ (resp. $J_1$) as the right most interval
of $\{I_a; a\in A_N\}$ (resp. $\{f(I_a); a\in A_N\}$). The Rauzy induction is seen as the first return map $f'$ induced by
$f$ on the interval $I\setminus J_0$ (resp. $I\setminus J_1$) if $|J_0|<|J_1|$ (resp. $|J_1|<|J_0|$).
We say the induction has type $1$ (resp. type $0$) if $|J_0|<|J_1|$ (resp. $|J_1|<|J_0|$).
The map $f'$ is also an IET.

If $\delta$ is the CIET associated to $f$, the Rauzy induction of $\delta$ is simply
the CIET $\delta'$ associated to $f'$. The benefit of using Rauzy inductions is that going from $\delta$ to $\delta'$
is easily done by using a Dehn twist (see section \ref{subsubsec:rauzyclasses}).

In section \ref{subsubsec:rauzyclasses}, we recall the definition of Rauzy classes (see \cite{Rau}, \cite{Yoc}).
Any CIET $\delta$ may be defined by a pair $(\pi, \lambda)$ where $\pi$ is a permutation and $\lambda$ is a length vector.
We define a partial order on permutation by saying a permutation $\pi'$ is a successor of a permutation $\pi$
if there exists a CIET $\delta'=(\pi', \lambda')$ which is the Rauzy induction of a CIET $\delta=(\pi, \lambda)$.
Each permutation has two successors (corresponding to the two types of Rauzy induction) and two predecessors.
This order can be represented by a graph called the Rauzy graph; connected components of such a graph are called Rauzy classes.
Any Rauzy induction correspond to an edge of a Rauzy class, and iterating the induction produces a path.\\

In section \ref{subsec:siciet}, we define self-induced CIETs. We say a CIET is self-induced if the sequence of Rauzy inductions
produces a periodic path in a Rauzy class.
The key part of this definition is that, taking a CIET $\delta$, the period of the Rauzy inductions produces an automorphism
$\varphi$ called the $\delta$-automorphism whose attracting subshift is equal to the $\delta$-subshift (theorem \ref{thm:othermainresult}) and which is
also decomposable into Dehn twists. This is especially important because the algorithm will essentially try to find such a decomposition.
This however, raises the following problem: there may be automorphisms $\psi$ such that $\Spsi=\Sv$ (their attracting subshifts are the same) which are not
decomposable along a path of a Rauzy class. Obviously, we still want the algorithm to be able to identify $\psi$
as an automorphism coming from a CIET. This issue is tackled in section \ref{subsubsec:diffaut}.\\

The algorithm is detailed in section \ref{subsec:finalalg}. It consists in attempting to decompose the automorphism
into Dehn twists along a path in a Rauzy class using a combinatorial Rauzy induction.
Theorem \ref{thm:othermainresult} (which says the subshift of a CIET $\delta$ is equal to the attracting subshift
of its $\delta$-automorphism) is especially important because it matches the singularities of the CIET with the singularities
of its associated automorphism, which is the basic idea behind section \ref{subsec:nececond}. In section \ref{subsec:nececond}, we give a list
of conditions on the attracting subshift (and in particular the singularities) of a positive primitive
automorphism that are necessary for Rauzy inductions to be possible.

The first step of the algorithm is to identify the singularities of the automorphism. This can be done
by using the algorithm of \cite{Jul}, which is briefly explained in appendix \ref{appendix:thebase}.
We then check that the condition of section \ref{subsec:nececond} are satisfied. If one of the conditions fails,
the algorithm stops and we conclude the automorphism does not come from an IET. If all the conditions
are satisfied, we apply the Rauzy induction. This induction produces
a Dehn twist which will be part of the decomposition (if it exists) of the automorphism. The algorithm then cycles
with a new and simpler automorphism and another set of singularities (obtained using the Dehn twist) until
the induction fails or the automorphism is fully decomposed.

An example is detailed in appendix \ref{appendix:example}.

\paragraph{Acknowledgements.} The author is indebted to Erwan Lanneau for valuable discussions, especially regarding
proposition \ref{prop:powerofvarphi}.

\section{Automorphisms of the free group}\label{sec:autom}

Let $F_N$ be the free group on $N\ge 2$ generators. Fixing an alphabet $A_N = \{a_0, \dots, a_{N-1}\}$ as a basis for $F_N$,
we denote $A_N^{-1} = \{a_0^{-1}, \dots, a_{N-1}^{-1}\}$ the set of inverse letters and consider $F_N$ to be the set of finite
words $v=v_0v_1\dots v_p$ with letters in $(A_N\cup A_N^{-1})$ and such that $v_i^{-1}\ne v_{i+1}$ for $0\le i < p$; such words are called \textbf{reduced}.
The identity element of $F_N$ will be called the empty word, and will be denoted $\epsilon$. The Gromov boundary $\partial F_N$ of $F_N$ is the set
of points $V=(V_i)_{i\in \mathds{N}}$ with letters in $(A_N\cup A_N^{-1})$ and such that $V_i^{-1}\ne V_{i+1}$ for any $i\in \mathds{N}$.
The \textbf{double boundary} $\partial^2F_N$ is defined by
\begin{center}
	$\partial^2F_N = (\partial F_N\times \partial F_N) \setminus \Delta$,
\end{center}
where $\Delta$ is the diagonal.
\begin{rmk}
We will refer to elements of $F_N$ as words, while elements of $\partial^2 F_N$
will be called points.
\end{rmk}

An automorphism $\varphi$ of $F_N$ is $\boldsymbol{A_N}$\textbf{-positive} if for all $a\in A_N$, all the letters of $\varphi(a)$ are in $A_N$.
When working with an $A_N$-positive automorphism, we always consider its representation over the free group endowed with
the basis $A_N$. The automorphism $\varphi$ induces a homeomorphism $\partial^2\varphi$ on $\partial^2 F_N$.
An $A_N$-positive automorphism is \textbf{primitive} if there is a positive integer $k$ such that all letters of $A_N$
are letters of $\varphi^k(a)$ for any $a\in A_N$.

Throughout this paper, $F_N$ will refer to the free group endowed with the basis $A_N$ and we always assume $N\ge 2$.\\

We define the shift map $S$ on $\partial^2 F_N$:
\begin{center}
	\begin{tabular}{cccc}
		$S~:$ & $\partial^2 F_N$ & $\to$ & $\partial^2 F_N$\\
		& $(X, Y)$ & $\mapsto$ & $(Y_0^{-1}X, Y_0^{-1}Y)$,
	\end{tabular}
\end{center}
 where $Y_0$ is the first letter of $Y$.

	\subsection{The attracting subshift}\label{subsec:attsub}
	Let $\varphi$ be an $A_N$-positive primitive automorphism and
	let $a$ be a letter of $A_N$. The primitivity condition implies that we can find an integer $k$
	such that $\varphi^k(a) = pas$ where $p$ and $s$ are non empty words of $F_N$ with letters in $A_N$.
	Now define
	\begin{center}
		$X = \lim\limits_{n\to +\infty} p^{-1}\varphi^k(p^{-1})\varphi^{2k}(p^{-1})\dots \varphi^{nk}(p^{-1})$,\\
		$Y = \lim\limits_{n\to +\infty} as\varphi^k(s)\varphi^{2k}(s)\dots \varphi^{nk}(s)$.
	\end{center}
	The \textbf{attracting subshift} of $\varphi$ is the set
	\begin{center}
		$\Sv = \overline{\{S^n(X, Y); n\in \mathds{Z}\}}$.
	\end{center}
	The map $S$ is a homeomorphism on $\Sv$.
	The attracting subshift only depends on $\varphi$ and not on the choice of the letter $a$ or the integer $k$
	(this is stated in different contexts in both \cite{Que} and \cite{BFH} (for example)).

	\subsection{Singularities of an automorphism}\label{subsec:autsing}
	Following from \cite{Jul}, we define the singularities of an automorphism.
	For any $w\in F_N$, the conjugacy $i_w$ is the automorphism of $F_N$ defined for any $v\in F_N$ by $i_w(v) = w^{-1}vw$.
	Let $\varphi$ be an $A_N$-positive primitive automorphism, and let $\Sv$ be its attracting subshift.
	\begin{defn}
	A $\boldsymbol{\varphi}$\textbf{-singularity} is a set $\Omega$ of (pairwise distinct) points of $\Sv$ satisfying the following conditions:
	\begin{itemize}
		\item $\Omega$ contains at least two elements,
		\item there exists an automorphism $\psi = i_w\circ \varphi^k$ for some conjugacy $i_w$ and some integer $k\ge 1$
		such that all the points of $\Omega$ are fixed points of $\partial^2 \psi$,
		\item for any $h\in \mathds{N}^*$, if $(U, V)\in \Sv$ is a fixed point of $\partial^2 \psi^h$,
		then $(U, V)\in \Omega$,
		\item there exist two points $(U, V)$ and $(U', V')$ of $\Omega$ such that we have either
		$V_0\ne V_0'$ or $U_0\ne U_0'$ or both.
	\end{itemize}
	We say that the automorphism $\psi$ fixes the singularity $\Omega$.
	\end{defn}

	These singularities are studied in great details in \cite{Jul}. In particular, it is explained how
	the index (\cite{GJLL}) of an automorphism (recall it affects directly the geometric representation
	of the automorphism's attracting subshift, see \cite{CH}) can be deduced from the singularities alone;
	it simply comes down to counting the number of points belonging to singularities
	(in \cite{Jul}, this index is referred to as the full outer index). The main result
	of \cite{GJLL} is that this index is bounded above by $N-1$. If $\varphi$ is
	an $A_N$-positive primitive automorphism, this result tells us that both the number of $\varphi$-singularities
	and the number of points belonging to $\varphi$-singularities are finite and bounded.

	It is not true that singularities are always caused by special infinite points; in our setting,
	a special right-infinite (resp. left-infinite) point $U$ (resp. $V$) is a point of $\partial F_N$
	for which there exist (at least) two points $V_{(0)}, V_{(1)}$ (resp. $U_{(0)}, U_{(1)}$) of $\partial F_N$
	with distinct first letters such that both $(U, V_{(0)})$ and $(U, V_{(1)})$ (resp. $(U_{(0)}, V)$
	and $(U_{(1)}, V)$) are in $\Sigma_\varphi$. Consider the following example.
	\begin{center}
		\begin{tabular}{ccccl}
		$\phi$ & $:$ & $a$ & $\mapsto$ & $abcad$\\
		&& $b$ & $\mapsto$ & $bd$\\
		&& $c$ & $\mapsto$ & $bc$\\
		&& $d$ & $\mapsto$ & $bca$\\
		\end{tabular}
	\end{center}
	Define
	\begin{center}
		\begin{tabular}{lcl}
			$U = \lim\limits_{n\to +\infty} \phi^{n} (c^{-1})$ && $V = \lim\limits_{n\to +\infty} \phi^{n} (a)$\\
			$U' = \lim\limits_{n\to +\infty} \phi^{2n} (a^{-1})$ && $V' = \lim\limits_{n\to +\infty} \phi^{n} (b)$
		\end{tabular}
	\end{center}
	and observe that both $(U, V)$ and $(U', V')$ are points of $\Sigma_\phi$ but that neither $(U, V')$ nor $(U', V)$ is:
	assuming $U_0, V_0, U_0',V_0'$ are the first letters of $U, V, U', V'$ respectively,
	this is done by checking that the words $U_0^{-1}V_0=ca$ and $U_0'^{-1}V_0'=ab$ are in the langage of $\phi$ (they are subwords of $\phi^n(a)$ for some $n$)
	but $U_0^{-1}V_0'=cb$ and $U_0'^{-1}V_0=aa$ are not. As there are no other points fixed by a power of $\phi$, we obtain a $\phi$-singularity $\Omega = \{(U, V), (U', V')\}$
	which is fixed by $\phi^2$.

	Note that even if the two points of $\Omega$ seem combinatorially unrelated, they will both correspond to the same point in the geometric representation
	of the attracting subshift (easily seen since the automorphism acts as a contracting homothety on the representation, see \cite{GJLL}).
	This should be considered the true meaning of singularities.

	In the general case of a positive primitive automorphism $\varphi$, this kind of behavior can only happen
	with singularities that are fixed by some power of $\varphi$. Specifically, it is proven in \cite{Jul} that all the points
	of a singularity fixed by an automorphism $i_w\circ \varphi^k$ with $w\ne \epsilon$ share a common coordinate.
	The behavior is most likely more widespread for non-positive automorphisms.

\section{Interval exchange transformations and completed interval exchange transformations}\label{sec:ciet}

Let $A_N$ be an alphabet with $N\ge 2$ letters. An \textbf{interval exchange transformation}
is a bijective map $f$ defined by:
\begin{itemize}
	\item a pair $\pi=(\pi_0, \pi_1)$ of bijections $\pi_0,\pi_1:A_N\to \{0, \dots, N-1\}$,
	\item a vector $\lambda = (\lambda_a)_{a\in A_N}$ of positive real numbers with $|\lambda| = \sum\limits_{a\in A_N} \lambda_a$,
	\item for any $a\in A_N$, define $\Lambda_0(a) = \sum\limits_{\pi_0(b) < \pi_0(a)} \lambda_b$
	and $\Lambda_1(a) = \sum\limits_{\pi_1(b) < \pi_1(a)} \lambda_b$, and
	\begin{center}
		\begin{tabular}{cccccl}
		$f$ & $:$ & $[0,~|\lambda|)$ & $\to$ & $[0,~|\lambda|)$ &\\
		&& $x$ & $\mapsto$ & $x - \Lambda_0(a) + \Lambda_1(a)$ & if $x\in [\Lambda_0(a),~\Lambda_0(a)+\lambda_a)$.
		\end{tabular}
	\end{center}
\end{itemize}
From now on, an interval exchange transformation will simply be called an \textbf{IET}.\\

We define the \textbf{completed interval exchange transformation} $\delta$ associated to $f$ as
the system of partial isometries $\delta = \{\delta_a; a\in A_N\}$ defined by:
\begin{center}
	\begin{tabular}{ccccc}
		$\delta_a$ & $:$ & $[\Lambda_0(a),~\Lambda_0(a)+\lambda_a]$ & $\to$ & $[\Lambda_1(a),~\Lambda_1(a)+\lambda_a]$\\
		&& $x$ & $\mapsto$ & $x - \Lambda_0(a) + \Lambda_1(a)$.
	\end{tabular}
\end{center}
We say that a point $x'$ is an \textbf{image} (resp. \textbf{pre-image}) of a point $x$ by $\delta$ if there exists
$a\in A_N$ such that $\delta_a(x) = x'$ (resp. $\delta_a^{-1}(x) = x'$).
It is important to notice that the isometries are defined on closed sets, effectively resulting in points with more
than one image (resp. pre-image).
From now on, a completed interval exchange transformation will simply be called a \textbf{CIET}. Also,
we will refer to the domain of any isometry $\delta_a$ by $D(\delta_a)$.

For convenience, we define $\delta_{a^{-1}} = \delta_a^{-1}$ for any $a\in A_N$. For any point $x\in [0,~|\lambda|]$, define
\begin{center}
	\begin{tabular}{cl}
		$\Sd(x) = \{(U, V)\in \partial^2 F_N; \forall n\in \mathds{N},$ & $U_n^{-1}\in A_N$ and $x\in D(\delta_{U_n}\circ \dots\circ \delta_{U_0})$\\
		& $V_n\in A_N$ and $x\in D(\delta_{V_n}\circ \dots\circ \delta_{V_0})\}$.
	\end{tabular}
\end{center}
The $\boldsymbol{\delta}$\textbf{-subshift} is the $S$-invariant set $\Sd = \bigcup\limits_{x\in [0,~|\lambda|]} \Sd(x)$.
The points $\Lambda_0(a)$ for $a\ne \pi_0^{-1}(0)$ are the \textbf{forward} $\boldsymbol{\delta}$\textbf{-singularities}
and the points $\Lambda_1(a)$ for $a\ne \pi_1^{-1}(0)$ are the \textbf{backward} $\boldsymbol{\delta}$\textbf{-singularities}.\\

From now on, defining an IET $f=(\pi, \lambda)$ and its associated CIET $\delta$ implicitely defines
the maps $\pi_0,\pi_1:A_N\to \{0, \dots N-1\}$, the points $\Lambda_0(a)$ and $\Lambda_1(a)$ for any $a\in A_N$,
the sets $\Sd(x)$ for any $x\in [0,~|\lambda|]$ and the $\delta$-subshift $\Sd$.

	\subsection{The Keane condition}\label{subsec:keane}
	An IET $f=(\pi, \lambda)$ is \textbf{minimal} (compare \cite{Kea1}) if for any $x\in [0, |\lambda|)$,
	the set $\{f^n(x); n\in \mathds{Z}\}$ is dense in $[0, |\lambda|)$.
	An IET $f=(\pi, \lambda)$ satisfies the Keane condition (see \cite{Kea1}, \cite{Kea2}) if
	for any $n\in \mathds{N}$ and any $a,b\ne \pi_0^{-1}(0)$, we have $f^n(\Lambda_0(a))\ne \Lambda_0(b)$ (there are no
	connections).
	It was proven in \cite{Kea1} that this condition suffices to ensure minimality. We study
	the effect of this condition on the associated CIET.
	\begin{prop}\label{prop:ckeane}
	Let $f=(\pi, \lambda)$ be an IET satisfying the Keane condition and let $\delta$ be the CIET associated to $f$.
	Then for any $x\in [0,~|\lambda|]$, the set $\Sd(x)$ contains at most two points.
	\end{prop}
	\begin{proof}
	Observe that the Keane condition implies that if $x$ is a forward (resp. backward)
	$\delta$-singularity, then for any integer $n\ge 0$, the point $f^{-n}(x)$ (resp. $f^n(x)$) cannot be a backward
	(resp. forward) $\delta$-singularity.

	The set $\Sd(x)$ can only contain more than one point if $x$ is contained in a backward orbit (under the action of the CIET)
	of a forward $\delta$-singularity or a forward orbit of a backward $\delta$-singularity.
	In other words, $\Sd(x)$ contains more than one point if $x$ is a $\delta$-singularity
	or if there exists a forward $\delta$-singularity $y$, a point $(U, V)\in \Sd(y)$ and an integer $m\ge 0$
	such that $x=\delta_{U_m}\circ \dots\circ \delta_{U_0}(y)$ or if there exists a backward singularity $z$,
	a point $(U, V)\in \Sd(z)$ and an integer $m\ge 0$ such that
	$x=\delta_{V_m}\circ \dots\circ \delta_{V_0}(z)$. Hence, there are only three situations that would produce
	points $x$ with more than two orbits (under the action of the CIET).

	\textit{Case $1$.} There exists a backward $\delta$-singularity $x$, a point $(U, V)$ of $\Sd(x)$ and
	an integer $m\ge 0$ such that $\delta_{V_m}\circ \dots\circ \delta_{V_0}(x)$ is a forward $\delta$-singularity.
	In this case, $x$ and any of the points $\delta_{V_i}\circ \dots\circ \delta_{V_0}(x)$ for $0\le i\le m$ would have
	at least four distinct orbits. Assuming $m$ is the smallest such integer, we obtain that $f^{m+1}(x)$
	is a forward $\delta$-singularity, which contradicts the Keane condition.

	\textit{Case $2$.} There exists a forward $\delta$-singularity $x$, a point $(U, V)$ of $\Sd(x)$ and
	an integer $m\ge 0$ such that $\delta_{V_m}\circ \dots\circ \delta_{V_0}(x)$ is a forward $\delta$-singularity.
	In this case, $x$ and any of the points $\delta_{U_i}\circ \dots\circ \delta_{U_0}(x)$ for $0\le i$ would have
	at least three distinct orbits. Observe that the point
	$\delta_{V_0}(x)$ is either the point $0$ or the point $|\lambda|$ (in which cases it is not a forward $\delta$-singularity)
	or a backward $\delta$-singularity (in which case the Keane condition also prevents it from being a forward $\delta$-singularity).
	Hence, $m$ cannot be $0$. Define $y=\delta_{V_1}\circ\delta_{V_0}(x)$ (resp. $y=\delta_{V_0}(x)$)
	if $\delta_{V_0}(x)$ is $0$ or $|\lambda|$ (resp. if $\delta_{V_0}(x)$ is neither $0$ nor $|\lambda|$) and
	observe $y$ is always a backward $\delta$-singularity. Assuming again that
	$m$ is the smallest integer such that $\delta_{V_m}\circ \dots\circ \delta_{V_0}(x)$ is a forward $\delta$-singularity,
	we obtain that $f^{m-1}(y)$ (resp. $f^m(y)$) is a forward $\delta$-singularity,
	which again contradicts the Keane condition.

	\textit{Case $3$.} There exists a backward $\delta$-singularity $x$, a point $(U, V)$ of $\Sd(x)$ and
	an integer $m\ge 0$ such that $\delta_{U_m}\circ \dots\circ \delta_{U_0}(x)$ is a backward $\delta$-singularity.
	In this case, $x$ and any of the points $\delta_{V_i}\circ \dots\circ \delta_{V_0}(x)$ for $0\le i$ would have
	at least three distinct orbits. This case is similar to case $2$.
	\end{proof}

	\begin{cor}\label{cor:ckeane}
	Let $x$ be a point with $\Sd(x) = \{(U, V), (U', V')\}$. Then we have either $U=U'$ or $V=V'$.
	\end{cor}

	We now state an important result of the present article. This theorem gives a simple
	definition of the $\delta$-subshift of a CIET $\delta$ associated to an IET $f$
	satisfying the Keane condition. The result can be found in \cite{Kea1}; we give an alternative
	proof.

	\begin{thm}[\cite{Kea1}]\label{thm:mainresult}
	Let $f=(\pi, \lambda)$ be an IET satisfying the Keane condition and let $\delta$ be the CIET associated to $f$.
	The set $\Sd(0)$ contains exactly one point $Z$ and we have
	$$\overline{\{S^n(Z); n\in \mathds{Z}\}} = \Sd.$$
	\end{thm}
	\begin{proof}
	We first prove that $\Sd(0)$ contains only one point. Since $0$ is not a $\delta$-singularity, then
	the point $f^{-1}(0)$ (resp. $f(0)$) is its only pre-image (resp. image) under the action of $\delta$.
	Observe that $f^{-1}(0)$ (resp. $f(0)$) is a forward (resp. backward) $\delta$-singularity.
	The Keane condition then implies that for any $i > 0$ the point $f^{-i}(0)$ (resp. $f^i(0)$)
	is not a backward (resp. forward) $\delta$-singularity and we conclude $\Sd(0)$ contains only one point $Z$.

	Let $W=(U, V)$ be a point of $\overline{\{S^n(Z); n\in \mathds{Z}\}}$. For any $n\in \mathds{N}$,
	define $D_n = D(\delta_{V_n}\circ \dots \circ \delta_{V_0})\cap D(\delta_{U_n}\circ \dots \circ \delta_{U_0})$
	(where $D$ denotes the domain);
	note that $D_n$ is a closed interval and that $D_{n+1}\subset D_n$.
	Moreover, for any $n\in \mathds{N}$, there exists a point
	$W'=(U', V')\in \{S^p(Z); p\in \mathds{Z}\}$ such that, for any $0\le i\le n$, we have
	$V_i' = V_i$ and $U_i'=U_i$, meaning $D_n$ is not empty. We conclude $\bigcap\limits_{n\in \mathds{N}} D_n$
	contains (at least) one point $x$ with $\Sd(x)=W$, and $\overline{\{S^n(Z); n\in \mathds{Z}\}}\subset \Sd$.

	Let $W=(U, V)$ be a point of $\Sd$. We define again
	$D_n = D(\delta_{V_n}\circ \dots \circ \delta_{V_0})\cap D(\delta_{U_n}\circ \dots \circ \delta_{U_0})$
	and we prove the following lemma.
	\begin{lem}
	For any $n\in \mathds{N}$, the set $D_n$ is a non trivial (not a singleton) closed interval.
	\end{lem}
	\begin{proof}
	It is obvious that $D_n$ is a non empty (since $W\in \Sd$) closed interval for any $n\in \mathds{N}$.
	Suppose $D_n$ is trivial for some $n$. We may simply assume there exists an integer $m\ge 2$
	such that $D(\delta_{V_m}\circ \dots \circ \delta_{V_0}) = \{x\}$ is trivial (if it is not
	the case, we can work with $S^k(W)$ for some $k\in \mathds{Z}$). Also assume
	$m$ is minimal in the sense that neither $D(\delta_{V_m}\circ \dots \circ \delta_{V_1})$
	nor $D(\delta_{V_{m-1}}\circ \dots \circ \delta_{V_0})$ is trivial.

	In that case, $x$ is either $0$ or $|\lambda|$ or a forward $\delta$-singularity.
	If it were not, we could choose a positive real number $\tau$ such that
	$(x-\tau,~x+\tau)\subset D(\delta_{V_0})$, and since $D(\delta_{V_m}\circ \dots \circ \delta_{V_1})$
	is not trivial and contains $\delta_{V_0}(x)$, the set $D(\delta_{V_m}\circ \dots \circ \delta_{V_0})$ would
	not be trivial.

	Similarly, the point $\delta_{V_{m-1}}\circ \dots \circ \delta_{V_0} (x)$ must also be
	either $0$ or $|\lambda|$ or a forward $\delta$-singularity. If it were not,
	we could again choose $\tau$ so that $(y-\tau,~y+\tau)\subset D(\delta_{V_m})$ with $y = \delta_{V_{m-1}}\circ \dots \circ \delta_{V_0}(x)$,
	resulting in a non trivial set $D(\delta_{V_m}\circ \dots \circ \delta_{V_0})$.

	This contradicts proposition \ref{prop:ckeane}.
	\end{proof}
	Recall from \cite{Kea1}, that since $f$ satisfies the Keane condition, $f$ is minimal. In particular,
	for any $n\in \mathds{N}$, there exists $k_n\in \mathds{N}$ such that $f^{k_n}(0)\in D_n$.
	We choose $k_n\ge n+1$ to avoid problems related to the fact that $f(0)$ is a backward
	$\delta$-singularity. Define $S^{k_n}(Z) = W^{(n)}$ for any $n\in \mathds{N}$, observe $(W^{(n)})_n$
	converges to $W$ and conclude.
	\end{proof}

	From now on, we only work with CIETs. A CIET $\delta$ will simply be defined as a pair $(\pi, \lambda)$,
	and following from proposition \ref{prop:ckeane}, we say $\delta$ verifies the Keane condition if
	$\Sd(x)$ contains at most two points for any $x\in [0,~|\lambda|]$.

	\subsection{Rauzy induction}\label{subsec:rauzyind}
	The following constructions are given (for IETs) in \cite{Rau}.
	Let $\delta=(\pi, \lambda)$ be a CIET. Define $\alpha_0 = \pi_0^{-1}(N-1)$ and $\alpha_1 = \pi_1^{-1}(N-1)$
	and assume $\lambda_{\alpha_0}\ne \lambda_{\alpha_1}$ (note this is automatically true
	if $\delta$ satisfies the Keane condition).

	Suppose $\lambda_{\alpha_0} > \lambda_{\alpha_1}$. We say $\delta$ has \textbf{type} $\boldsymbol{0}$.
	The \textbf{completed Rauzy induction} of $\delta$ is the CIET $\delta'=(\pi', \lambda')$ defined by
	\begin{itemize}
		\item $\pi_0' = \pi_0$,
		\begin{itemize}
			\item $\forall a\in A_N;~\pi_1(a)\le \pi_1(\alpha_0),~\pi_1'(a) = \pi_1(a)$,
			\item $\pi_1'(\alpha_1) = \pi_1(\alpha_0)+1$,
			\item $\forall a\in A_N;~\pi_1(\alpha_0) < \pi_1(a) < \pi_1(\alpha_1),~\pi_1'(a)=\pi_1(a)+1$.
		\end{itemize}
		\item $\lambda_{\alpha_0}' = \lambda_{\alpha_0}-\lambda_{\alpha_1}$ and $\forall a\ne \alpha_0,~\lambda_a'=\lambda_a$.
	\end{itemize}
	If we simply name $I_a'$ the domain of $\delta_a'$ for any $a\in A_N$, then 
	$\delta_{\alpha_1}'(I_{\alpha_1}') = \delta_{\alpha_0}\circ \delta_{\alpha_1}(I_{\alpha_1}')$
	and $\delta_a'(I_a') = \delta_a(I_a')$ for any $a\ne \alpha_1$.
	\begin{figure}[h!]
	\begin{center}
		\scalebox{1}{\includegraphics{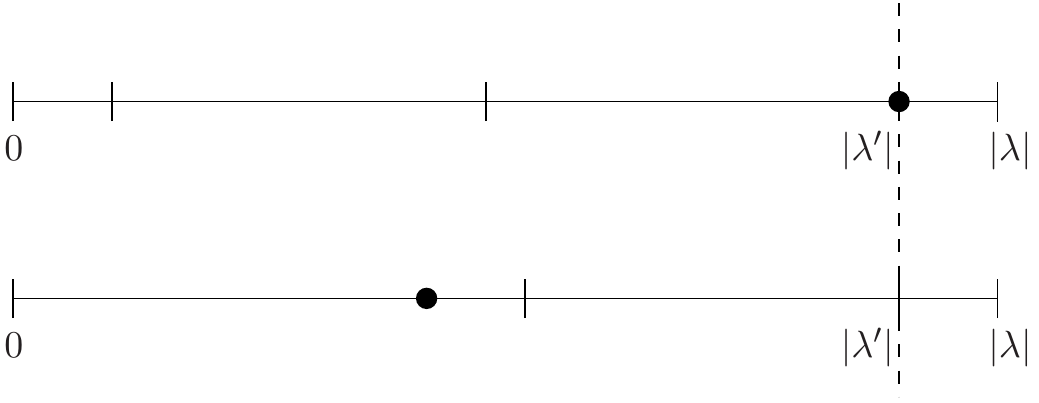}}
	\end{center}
	\caption{An example of type $0$ induction.}
	\label{fig:type0}
	\end{figure}

	If $\lambda_{\alpha_0} < \lambda_{\alpha_1}$, then we say $\delta$ has \textbf{type} $\boldsymbol{1}$.
	The \textbf{completed Rauzy induction} of $\delta$ is the CIET $\delta'=(\pi', \lambda')$ defined by
	\begin{itemize}
		\item $\pi_1' = \pi_1$,
		\begin{itemize}
			\item $\forall a\in A_N;~\pi_0(a)\le \pi_0(\alpha_1),~\pi_0'(a) = \pi_0(a)$,
			\item $\pi_0'(\alpha_0) = \pi_0(\alpha_1)+1$,
			\item $\forall a\in A_N;~\pi_0(\alpha_1) < \pi_0(a) < \pi_0(\alpha_0),~\pi_0'(a)=\pi_0(a)+1$.
		\end{itemize}
		\item $\lambda_{\alpha_1}' = \lambda_{\alpha_1}-\lambda_{\alpha_0}$ and $\forall a\ne \alpha_1,~\lambda_a'=\lambda_a$.
	\end{itemize}
	Simply naming $I_a'$ the domain of $\delta_a'$ for any $a\in A_N$, we have
	$\delta_{\alpha_0}'(I_{\alpha_0}') = \delta_{\alpha_0}\circ \delta_{\alpha_1}(I_{\alpha_0}')$
	and $\delta_a'(I_a') = \delta_a(I_a')$ for any $a\ne \alpha_0$.
	\begin{figure}[h!]
	\begin{center}
		\scalebox{1}{\includegraphics{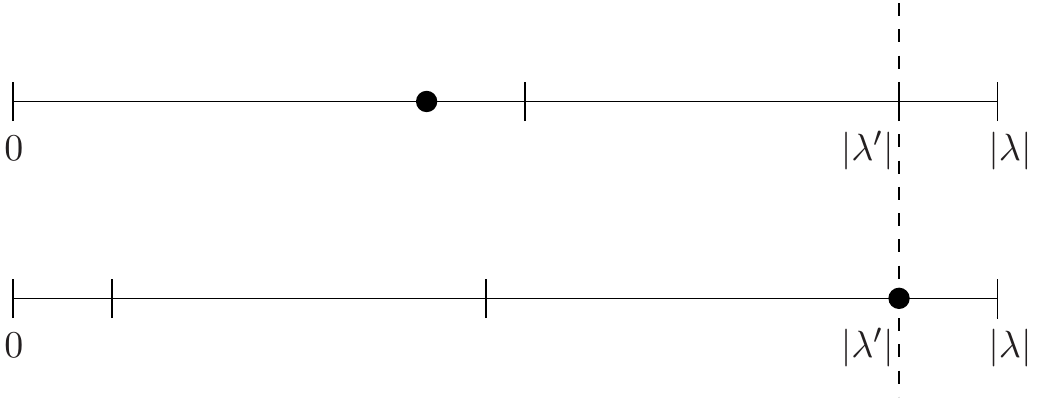}}
	\end{center}
	\caption{An example of type $1$ induction.}
	\label{fig:type1}
	\end{figure}

	We write $\delta' = R(\delta)$ in both cases.

	\begin{rmk}\label{rmk:firstreturn}
	Observe that the induction can be seen as a first return system on $[0,~|\lambda'|]$ except for one point: the left end point of
	$I_{\alpha_1}'$ (resp. $I_{\alpha_0}'$) if $\delta$ has type $0$ (resp. type $1$).
	\end{rmk}

		\subsubsection{Rauzy classes}\label{subsubsec:rauzyclasses}
		A pair $\pi=(\pi_0, \pi_1)$ is called \textbf{reducible} if $\pi_1\circ\pi_0^{-1}(\{0, \dots, k\}) = \{0, \dots, k\}$ for some
		$k<N-1$. Irreducibility is a consequence of the Keane condition, and in fact, we are only interested in irreducible pairs.

		Given irreducible pairs $\pi=(\pi_0, \pi_1)$ and $\pi'=(\pi_0', \pi_1')$ where $\pi_0, \pi_1, \pi_0', \pi_1':A_N\to \{0, \dots, N-1\}$
		are bijections, we say that $\pi'$ is a successor of $\pi$ if there exist two vectors $\lambda, \lambda'$ of $\mathds{R}^N$ with positive entries
		such that $R(\pi, \lambda) = (\pi', \lambda')$. Any pair $\pi$ has exactly two successors, corresponding to types $0$ and $1$,
		and any pair $\pi'$ is the successor of exactly two pairs. This relation defines a partial order on the sets of irreducible pairs
		that may be represented by a directed graph called the \textbf{Rauzy graph}; obviously, there is one graph per value of $N$.
		The connected components of such graphs are called \textbf{Rauzy classes}. An example is given on figure \ref{fig:rauzy}.

		We now assume $\delta=(\pi, \lambda)$ is a CIET satisfying the Keane condition. Observe then that for any $n\in \mathds{N}$,
		the map $R^n(\delta)$ is defined and that $\delta$ determines an infinite path in one of the Rauzy classes.
		Moreover, each edge of the Rauzy graph comes with an elementary automorphism (a Dehn twist) intended to specify
		the transition from one subshift to the other (see figure \ref{fig:rauzy}).
		Define $R(\delta) = \delta' = (\pi', \lambda')$ and $\alpha_0 = \pi_0^{-1}(N-1), \alpha_1 = \pi_1^{-1}(N-1)$.
		If $\delta$ has type $0$, define the automorphism $\sigma$ of $F_N$ by
		\begin{center}
			\begin{tabular}{cccccl}
			$\sigma$ & $:$ & $\alpha_1$ & $\mapsto$ & $\alpha_1\alpha_0$ &\\
			&& $a$ & $\mapsto$ & $a$ & if $a\ne \alpha_1$.
			\end{tabular}
		\end{center}
		If $\delta$ has type $1$, define the automorphism $\sigma$ of $F_N$ by
		\begin{center}
			\begin{tabular}{cccccl}
			$\sigma$ & $:$ & $\alpha_0$ & $\mapsto$ & $\alpha_1\alpha_0$ &\\
			&& $a$ & $\mapsto$ & $a$ & if $a\ne \alpha_0$.
			\end{tabular}
		\end{center}
		We obtain the following proposition.
		\begin{prop}\label{prop:recoding}
		For any point $x\in [0,~|\lambda'|)$, we have $\Sd(x) = \{\partial^2\sigma(W); W\in \Sdd(x)\}$.
		The set $\Sdd(|\lambda'|)$ contains only one point $W$ and $\partial^2 \sigma(W)\in \Sd(|\lambda'|)$.
		\end{prop}
		Note that $\Sd(|\lambda'|)$ contains two points. However, the point $|\lambda'|$
		is the only point of $[0,~|\lambda'|]$ such that $\Sd(|\lambda'|)$ contains a point $(U, V)$ with
		$U_0 = \alpha_1^{-1}$ and $V_0= \alpha_0$. This means $(U', V') = \partial^2 \sigma^{-1}(U, V)$ cannot be a point
		of $\Sdd$; in particular the first letter of $U'$ (resp. $V'$) if $\delta$ has type $0$ (resp. type $1$)
		is in $A_N$ (resp. $A_N^{-1}$).

		We end this section with an example of Rauzy class, complete with the Dehn twists associated with each transition.\\
		\begin{figure}[h!]
		\begin{center}
			\scalebox{0.9}{\includegraphics{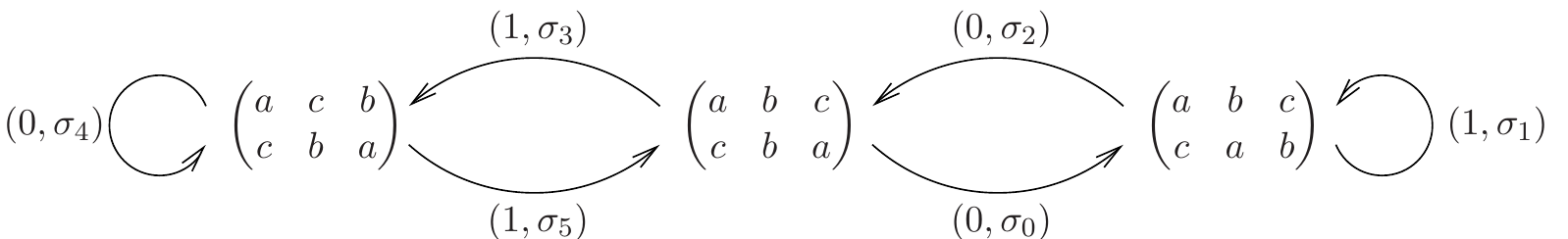}}
		\end{center}
		\caption{An example of Rauzy class.}
		\label{fig:rauzy}
		\end{figure}\\
		Each edge is labeled with a pair $(\tau, \sigma)$ representing the type and the associated recoding automorphism;
		these automorphisms are defined below.
		\begin{center}
		\begin{tabular}{llllllllllll}
			$\sigma_0:$ & $a\mapsto ac$ & $\sigma_1:$ & $a\mapsto a$ & $\sigma_2:$ & $a\mapsto a$ & $\sigma_3:$ & $a\mapsto a$ & $\sigma_4:$ & $a\mapsto ab$ & $\sigma_5:$ & $a\mapsto a$\\
			& $b\mapsto b$ && $b\mapsto b$ && $b\mapsto bc$ && $b\mapsto b$ && $b\mapsto b$ && $b\mapsto ab$\\
			& $c\mapsto c$ && $c\mapsto bc$ && $c\mapsto c$ && $c\mapsto ac$ && $c\mapsto c$ && $c\mapsto c$\\
		\end{tabular}
		\end{center}

		\subsubsection{Rauzy induction and singularities}
		In this section, we study the evolution of the set of singularities under the action of the Rauzy induction.
		We assume $\delta = (\lambda, \pi)$ with $\pi=(\pi_0, \pi_1)$ and $\pi_0, \pi_1: A_N\to \{0, \dots, N-1\}$ is a CIET which satisfies the Keane condition and we define
		$\delta'=R(\delta)$, $\alpha_0 = \pi_0^{-1}(N-1)$ and $\alpha_1 = \pi_1^{-1}(N-1)$. Observe that if $S_\delta$ (resp. $S_{\delta'}$) is the set
		of $\delta$-singularities (resp. $\delta'$-singularities) then $S_{\delta'}\setminus S_\delta$ (resp. $S_\delta\setminus S_{\delta'}$)
		contains exactly one point $x$ (resp. $y$).
		
		The point $y$ is obviously the right most $\delta$-singularity; it is the left end point of the domain of $\delta_{\alpha_1^{-1}}$
		(resp. $\delta_{\alpha_0}$) if $\delta$ has type $0$ (resp. $1$).

		By definition of the induction, the point $x$ is the left end point of the domain of $\delta'_{\alpha_1^{-1}}$ (resp. $\delta'_{\alpha_0}$)
		and we have $x = \delta_{\alpha_0}(y)$ (resp. $x = \delta_{\alpha_1^{-1}}(y)$) if $\delta$ has type $0$ (resp. $1$).

		Observe that, as a consequence of the Keane condition, the point $x$ only has one image and one pre-image by $\delta$ (regardless of type).
		The following proposition summarizes the discussion above and will be useful later.
		\begin{prop}\label{prop:newsing}
		If $\delta$ has type $0$ (resp. type $1$), then the pre-image (resp. image) by $\delta$ of the left end point of the domain
		of $\delta_{\alpha_1^{-1}}'$ (resp. $\delta_{\alpha_0}'$) is the left end point of the domain of $\delta_{\alpha_1^{-1}}$
		(resp. $\delta_{\alpha_0}$).
		\end{prop}

		We deduce that the pre-image (resp. image) of $x$ by $\delta$ is a backward (resp. forward) $\delta$-singularity if $\delta$ has type $0$ (resp. type $1$),
		meaning $\Sd(x) = \{(U, V), (U', V')\}$ will always contain two points. Observe that we have
		$U_0U_1 = \alpha_0^{-1}\alpha_1^{-1}$ and $U_0'U_1'= \alpha_0^{-1}\beta_0^{-1}$ with $\pi_1(\beta_0)=N-2$
		(resp. $V_0V_1 = \alpha_1\alpha_0$ and $V_0'V_1'= \alpha_1\beta_1$ with $\pi_0(\beta_1)=N-2$)
		if $\delta$ has type $0$ (resp. type $1$), and conclude $\Sdd(x) = \{\partial^2 \sigma^{-1}(U, V), \partial^2 \sigma^{-1}(U', V')\}$
		(where $\sigma$ is the Dehn twist associated with the induction, as defined in the previous section) effectively is a $\delta'$-singularity.

	\subsection{Self-induced CIETs}\label{subsec:siciet}

	We now relate $A_N$-positive primitive automorphisms with CIETs by introducing self-induced CIETs.
	We say that a CIET $\delta^{(0)}=(\pi^{(0)}, \lambda^{(0)})$ is \textbf{self-induced} if there exists a positive integer $n$ such
	that $R^n(\delta^{(0)}) = (\pi^{(n)}, \lambda^{(n)})$ with 
	\begin{itemize}
		\item $\pi^{(n)} = \pi^{(0)}$,
		\item $\lambda^{(n)} = \eta\lambda^{(0)}$ for some positive real number $\eta$.
	\end{itemize}
	Alternatively, one can say a CIET is self-induced if its associated path in the Rauzy graph
	is periodic. The path $(\pi^{(0)}, \pi^{(1)}, \dots, \pi^{(n)}=\pi^{(0)})$ is a cycle of the Rauzy graph.
	Assume it is minimal: it is not a power of a smaller cycle. Suppose $\sigma_i$ is the automorphism
	associated (as in the previous section) to the edge $(\pi^{(i)}, \pi^{(i+1)})$ for any $0\le i < n$.

	We define the $\boldsymbol{\delta}$\textbf{-automorphism} $\varphi$ as the automorphism $\varphi = \sigma_0\circ \dots\circ \sigma_{n-1}$.
	The automorphism $\varphi$ is obviously $A_N$-positive. Moreover,
	we get from \cite[corollary 3]{Yoc} that $\delta$ satisfies the Keane condition, and from
	\cite[corollary 4]{Yoc} that $\varphi$ is primitive.

	\begin{thm}\label{thm:othermainresult}
	Let $\delta = (\pi, \lambda)$ be a self-induced CIET and let $\varphi$ be the associated $\delta$-automorphism.
	Define the $\delta$-subshift $\Sd$ and the attracting subshift $\Sv$ of $\varphi$. We have
	$\Sd = \Sv$.
	\end{thm}
	\begin{proof}
	Recall that since $\delta$ satisfies the Keane condition, the set $\Sd(0)$ contains exactly one point $Z$.
	Thanks to theorem \ref{thm:mainresult}, we only need to prove $Z$ is in $\Sv$.

	Let $n$ be the smallest positive integer such that $R^n(\delta) = \delta' = (\pi', \lambda')$ with
	$\pi' = \pi$ and $\lambda' = \eta\lambda$ ($\eta>0$). Observe that $\Sd(0) = \Sdd(0) = \{Z\}$.
	We deduce from proposition \ref{prop:recoding} that $\partial^2 \varphi (Z) = Z$.
	Define $Z = (X, Y)$. Considering any letter $Y_n$ (resp. $X_n^{-1}$) is an element of $A_N$ and since $\varphi$ is $A_N$-positive
	primitive, it is easy to see that $Y = \lim\limits_{n\to +\infty} \varphi^n(Y_0)$ and $X = \lim\limits_{n\to +\infty} \varphi^n(X_0)$.
	To complete the proof, we also need to prove that the word $X_0^{-1}Y_0$ is contained in $\varphi^k(a)$ for some $k\ge 1$ and $a\in A_N$.
	By definition, $Y_0 = \pi_0^{-1}(0)$ and $X_0^{-1} = \pi_1^{-1}(0)$. Observe $D(\delta_{Y_0})\cap D(\delta_{X_0})$
	(where $D$ denotes the domain) is not trivial.
	By minimality, there exists a point $x$ in $D(\delta_{Y_0})\cap D(\delta_{X_0})$ and a positive integer $n$ such that
	$S^n(Z)\in \Sd(x)$, and we deduce the word $X_0^{-1}Y_0$ is contained in $\varphi^k(Y_0)$ for some $k$.
	We conclude $Z$ is a point of $\Sv$ and $\Sd = \Sv$.
	\end{proof}

	The ability to decompose a $\delta$-automorphism by following a path in a Rauzy class is one of the key ideas of the algorithm.
	Given a positive primitive automorphism $\varphi$, the algorithm will attempt to decompose $\varphi$ into
	Dehn twists by applying a combinatorial version of the Rauzy induction on its attracting subshift $\Sv$
	using the $\varphi$-singularities. An important issue is that multiple automorphisms may have the same attracting
	subshift (and therefore the same singularities) but may not all be decomposable along a path in a Rauzy class.
	Starting with an automorphism $\varphi$, one of the first step of the algorithm will be to choose an automorphism
	$\psi$ with $\Spsi=\Sv$ which we know will be a viable candidate for the decomposition.
	To that end, we study automorphisms that share their attracting subshifts with a $\delta$-automorphism in the following
	sections.

		\subsubsection{Singularities of a $\delta$-automorphism}

		First, we relate the singularities of the CIET $\delta$ with the singularities of the $\delta$-automorphism.

		Let $\delta$ be a self-induced CIET and $\varphi$ its $\delta$-automorphism.
		Observe that for any $\delta$-singularity $x$, the set $\Sd(x)$ is a $\varphi$-singularity:
		we can easily deduce that there must exist an automorphism $i_w\circ \varphi^k$
		such that $\partial^2 (i_w\circ \varphi^k)$ fixes all the points of $\Sd(x) = \{(U, V), (U', V')\}$
		by noticing that since $U=U'$ (resp. $V=V'$) if $x$ is a forward (resp. backward) $\delta$-singularity,
		we also have $\partial^2\varphi(U, V) = (X, Y)$ and $\partial^2\varphi(U', V') = (X', Y')$ with $X=X'$ (resp. $Y=Y'$).
		In fact, this is also true for $k=1$.
		\begin{prop}\label{prop:invsing}
		For any singularity $\Omega$, there exists $w\in F_N$ such that $i_w\circ \varphi$ fixes $\Omega$.
		\end{prop}
		\begin{proof}
		Define $\delta = (\pi, \lambda)$. Let $n$ be the smallest integer such that
		$R^n(\delta) = \delta' = (\pi', \lambda')$ with $\pi'=\pi$ and $\lambda' = \eta\lambda$
		for some positive real number $\eta$. Recall
		(proposition \ref{prop:recoding}) that for any point $x\in [0, |\lambda'|)$, we have
		$\Sd(x) = \partial^2\varphi(\Sdd(x))$. We deduce from $\pi=\pi'$ together with proposition \ref{prop:newsing}
		that if $x$ is the $i$th (from the left) forward (resp. backward) $\delta'$-singularity, then there exists a
		word $u = u_0u_1\dots u_p$ with letters in $A_N$ (resp. $A_N^{-1}$) such that
		$\delta_{u_p}\circ \dots\circ \delta_{u_0}(x)$ is the $i$th (from the left) forward (resp. backward) $\delta$-singularity.
		This last property tells us any $\varphi$-singularity $\Omega$ is globally invariant by $i_w\circ \varphi$
		for some $w$: we have $\partial^2(i_w\circ \varphi) (\Omega) = \Omega$. We conclude each point of $\Omega$
		is in fact fixed by using $\pi=\pi'$ again.
		\end{proof}
		Note that this property is quite specific to $\delta$-automorphisms. In the general case, we often need to consider
		powers of a given automorphism to fix its singularities.

		\subsubsection{Different positive primitive automorphisms may have the same attracting subshift}\label{subsubsec:diffaut}

		Following from \cite{BFH}, if $\varphi$ and $\psi$ are two $A_N$-positive primitive automorphism with
		$i_u\circ \psi^k = \varphi^h$ for some $u\in F_N$ and $k,h\ge 1$, then $\Sv = \Spsi$. It is unclear in the general case if
		other automorphisms sharing the attracting subshift may be found.

		Our algorithm will attempt to decompose a given automorphism $\varphi$ along paths of Rauzy classes
		in order to determine if it comes from a CIET. The issue is that such a decomposition may not exist for $\varphi$,
		but may exist for another automorphism $\psi$ with a similar attracting subshift.
		In proposition \ref{prop:powerofvarphi}, we study automorphisms which share their attracting subshift with the $\delta$-automorphism
		of a CIET $\delta$. This will provide us with a way to determine which automorphisms are good candidates for the decomposition.

		\begin{prop}\label{prop:powerofvarphi}
		Let $\psi$ is an $A_n$-positive automorphism such that $\Spsi=\Sv$ and let $k$
		be a positive integer such that any $\varphi$-singularity $\Omega$ is fixed by $i_w\circ \psi^k$ for some $w\in F_N$.
		Then there exists $u\in F_N$ such that $\partial^2 (i_u\circ\psi^k) (Z)=Z$ and we have
		$i_u\circ\psi^k = \varphi^h$ for some positive integer $h$.
		\end{prop}
		\begin{proof}
		First, note that the set of $\psi$-singularities is the set of $\varphi$-singularities: this is obvious since
		$\Spsi = \Sv$. Define $Z = (X, Y)$ and observe there exists a backward $\varphi$-singularity containing the point
		$(U, V) = Y_0^{-1}Z$ (where $Y_0$ is the first letter of $Y$). We know there exists $v\in F_N$ such that
		$\partial^2 (i_v\circ \psi^k) (U, V) = (U, V)$ and we deduce $\partial^2 (i_u\circ \psi^k) (Z) = Z$
		for $u = \psi^k(Y_0)vY_0^{-1}$.

		In order to prove that $i_u\circ\psi^k = \varphi^h$ for some $h>0$, we are going to use results on
		pseudo-Anosov homeomorphisms on translation surfaces.
		We refer to \cite{Vee} for the original construction of pseudo-Anosov homeomorphisms using interval
		exchange transformations and to \cite{BL} for a nice summary of the necessary definitions and results.
		General results on homeomorphisms on surfaces may also be found in \cite{FLP}.

		A homeomorphism $f$ on a surface $\mathcal{S}$ (we assume it has a boundary component) is pseudo-Anosov if there exist two transversely measured
		foliations $(\mathcal{F}_s, \mu_s)$ and $(\mathcal{F}_u, \mu_u)$, respectively called the stable
		and unstable foliations, such that $f(\mathcal{F}_s) = \eta^{-1}\mathcal{F}_s$ and
		$f(\mathcal{F}_u) = \eta\mathcal{F}_u$ for some real number $\eta>1$. The action of the map $f$ on the fundamental group of $S$
		can be seen as a free group automorphism whose attracting subshift is a combinatorial interpretation
		of the stable foliation.

		We associate (as in \cite{Vee} or \cite{BL}) a translation surface $S_\delta$ to our self-induced CIET $\delta$.
		In \cite{Vee}, Veech gives an interpretation on $S_\delta$ of the Rauzy induction; the transformation is
		referred to as the Rauzy-Veech induction.
		The following result is attributed to Veech in \cite{BL}. It should be understood up to composition by a conjugacy.
		\begin{thm}[\cite{Vee}]\label{thm:pA}
			All pseudo-Anosov homeomorphisms of $S_\delta$ that fix a separatrix are obtained by Rauzy-Veech inductions.
		\end{thm}
		A separatrix is a half leaf of the stable foliation that is attached to a singularity of $S_\delta$.
		A point $x\in S_\delta$ belonging to a leaf $L$ is a singularity if the combinatorial interpretation of $L$
		with marked point $x$ is a point $(U', V')$ of a $\varphi$-singularity; the left end point and the right end point
		of the interval are also considered singularities. Hence, a separatrix 
		is interpreted combinatorially as a point $X'\in \partial F_N$ such that there exists $Y'\in \partial F_N$
		such that $(X', Y')$ (resp. $(Y', X')$) belongs to a $\varphi$-singularity or is the orbit of the
		left end point or right end point of the interval; such a point is called a combinatorial separatrix.

		It is possible that the homeomorphism $\partial \psi^k$ (induced on $\partial F_N$ by $\psi^k$) does not fix a combinatorial separatrix.
		In that case, we prove that we can work with an $A_N$-positive automorphism $\psi' = i_v\circ \psi^k$ that fixes
		a $\varphi$-singularity ($\partial \psi'$ will then fix combinatorial separatrix).

		Recall we assumed any $\varphi$-singularity is fixed by $i_w\circ \psi^k$ for some $w$.
		Suppose $\partial \psi^k$ does not fix a combinatorial separatrix. Then $\psi^k$ does not fix a $\varphi$-singularity.
		We also deduce that there must exist a letter $a_0\in A_N$ such that $a_0$ is the first letter of $\psi^k(a_0)$ (otherwise
		there would be a $\varphi$-singularity that is fixed by some power of $\psi^k$ but not by $\psi^k$). Moreover, 
		if $i_w\circ \psi^k$ fixes the $\varphi$-singularity $\Omega = \{(X, Y), (X, Y')\}$, then the first letter of $w$
		is also the first letter of both $\psi^k(Y_0)$ and $\psi^k(Y_0')$. Applying this to all $\varphi$-singularities,
		we deduce $a_0$ is the first letter of $\psi^k(a)$ for any $a\in A_N$. Similarly, one can prove there exist
		a letter $b_0$ which is the last letter of $\psi^k(a)$ for any $a\in A_N$. We can then obtain an $A_N$-positive
		automorphism $\psi' = i_v\circ \psi^k$ that fixes a $\varphi$-singularity.

		We simply assume $\psi^k$ fixes a singularity. Both $\varphi$ and $\psi^k$ are seen as the action of pseudo-Anosov
		homeomorphisms fixing separatrix on the fundamental group of $S_\delta$, and we deduce from theorem \ref{thm:pA} that
		there exist two positive integers $i$ and $j$ such that $(i_u\circ\psi^k)^i = \varphi^j$.
		We conclude with the definition of the $\delta$-automorphism and proposition \ref{prop:invsing}.
		\end{proof}

		Effectively, this proposition tells us the decomposition should be attempted on automorphisms which fix all singularities, and which also
		fix a certain point $Z$ that will correspond to the left most point of the interval.

		\subsubsection{Rauzy inductions happen on the right}\label{subsubsec:mirror}

		Our definition of self-induction is somewhat non canonical because the Rauzy induction always happens on the right.
		As we mentioned, the algorithm will attempt a combinatorial version of the Rauzy induction using only
		the attracting subshift of a given automorphism. As the combinatorial nature of the process does
		not provide a clear definition of left and right, the very first induction will require an arbitrary
		choice for the right side. The aim of the present section is to give us an easy way to make this choice.

		Consider a self-induced CIET $\delta = (\pi, \lambda)$. Define the CIET $\delta' = (\pi', \lambda)$ such that,
		for any $a\in A_N$, we have $\pi_0'(a) = N-1-\pi_0(a)$ and $\pi_1'(a) = N-1-\pi_1(a)$.
		We will say that $\delta'$ is the \textbf{mirror} of $\delta$.
		The CIETs $\delta$ and $\delta'$ are essentially the same (if $\delta$ is the CIET on the interval $[0, |\lambda|]$,
		then $\delta'$ can be seen as the same CIET on $[|\lambda|, 0]$)
		and we have $\Sdd = \Sd$. However, $\delta'$ may not be auto-induced in our sense.

		For example, let $\eta$ be the greatest root of $\eta^2-4\eta+1$. Define
		$\lambda_a=2\eta-1$, $\lambda_b=\eta$ and $\lambda_c=2\eta$. Also define
		the permutations
		\begin{center}
			$\pi = \begin{pmatrix}a&b&c\\c&a&b\end{pmatrix}$ \hspace{1cm} and \hspace{1cm} $\pi' = \begin{pmatrix}c&b&a\\b&a&c\end{pmatrix}$
		\end{center}
		and the CIETs $\delta = (\pi, \lambda)$ and $\delta' = (\pi', \lambda)$.
		We obtain $R^5(\delta) = (\pi, \eta^{-1}\lambda)$, so $\delta$ is effectively self-induced.
		However, for any positive integer $n$, there will never be a real number $\eta'$ such that
		$R^n(\delta')$ is the CIET $(\pi', \eta'\lambda)$: a simple iteration shows that
		we obtain $R^5(\delta') = (\pi'', \lambda'')$ with
		\begin{center}
			$\pi'' = \begin{pmatrix}c&a&b\\b&a&c\end{pmatrix}$ \hspace{1cm} and \hspace{1cm} $\lambda'' = \begin{pmatrix}1\\\eta-2\\1\end{pmatrix}$
		\end{center}
		and $R^5(\pi'', \lambda'') = (\pi'', \eta^{-1}\lambda'')$.

		It is reasonable to think this is the general behavior of self-induced CIETs. Namely, if $\delta$ is a self-induced CIET and we define
		$\delta'$ as above, then the path of the Rauzy graph associated to $\delta'$ is eventually periodic.

		\begin{prop}\label{prop:mirror}
		Let $\delta=(\pi, \lambda)$ be a self-induced CIET and let $\delta'=(\pi', \lambda)$
		be its mirror. Define $\varphi$ as the $\delta$-automorphism, $\Sd(0) = \{Z\}$ and $\Sd(|\lambda|) = \Sdd(0) = \{Z'\}$.
		The CIET $\delta'$ is self-induced if and only if the automorphism $i_u\circ\varphi$ such that
		$\partial^2 (i_u\circ\varphi)(Z')=Z'$ is $A_N$-positive.
		\end{prop}
		\begin{proof}
		Note that the existence of a word $u\in F_N$ such that $\partial^2 (i_u\circ\varphi)(Z')=Z'$
		is a consequence of proposition \ref{prop:invsing}. If $\delta'$ is self-induced,
		we deduce $i_u\circ\varphi$ is $A_N$-positive from the definitions of the $\delta$-automorphism
		and $\delta'$-automorphism along with proposition \ref{prop:powerofvarphi}.

		Assume $\psi = i_u\circ\varphi$ is $A_N$-positive and consider the CIET $\delta'$ on the interval $[0, |\lambda|]$.
		For any $a\in A_N$, define $I_a\subset [0, |\lambda|]$ (resp. $I_{a^{-1}}\subset [0, |\lambda|]$) as the set of points $x$ such that
		there exists $(U, V)\in \Sdd$ with $V_0=a$ (resp. $U_0=a^{-1}$) such that $\partial^2\psi(U, V)\in \Sdd(x)$.
		Define the system of partial isometries $\delta'' = (\delta_a'')_{a\in A_N}$ where for any $a\in A_N$, the map $\delta_a'':I_a\to I_{a^{-1}}$
		is a translation. We want to show that $\delta'' = R^n(\delta')$ for some positive integer $n$.
		First, we prove $\delta''$ is effectively a CIET. For any $a\in A_N$, the set $I_a$ (resp. $I_{a^{-1}}$) is a closed interval;
		if it were not, the domain of $\delta_a'$ (resp. $\delta_{a^{-1}}'$) would not be a closed interval either.
		Define $Z'=(X', Y')$. Since $\partial^2\psi (Z')=Z'$, both $I_{Y_0}$ and $I_{X_0}$ contain $0$. We deduce from
		proposition \ref{prop:invsing} that for any $a,b\in A_N$, $I_a\cap I_b$ (resp. $I_{a^{-1}}\cap I_{b^{-1}}$) contains (exactly) one point
		if and only if $|\pi_0(a)-\pi_0(b)|=1$ (resp. $|\pi_1(a)-\pi_1(b)|=1$) and is empty otherwise. Hence, $\delta''$
		is a CIET and we have $\delta'' = (\pi'', \lambda'')$ with $\pi''=\pi'$. Moreover, we have
		$\lambda''=\eta\lambda'$ for some positive real number $\eta$; if this were not the case, we would easily deduce $\Sigma_{\delta''}(0)\ne \{Z'\}$
		from minimality. Also, as a consequence of \cite[proposition 6.2]{CanSie}, $\delta''$ is a first return system (with respect to remark \ref{rmk:firstreturn}).
		We conclude by using \cite[proposition 8.9]{Vee} and \cite[theorem 23]{Rau} which state a first return system on $[0, |\lambda''|]$ has (exactly) $N-1$
		singularities if and only if it can be obtained by Rauzy inductions.
		\end{proof}

\section{The algorithm}

Theorem \ref{thm:othermainresult} is especially useful because it allows us to translate obvious geometric
properties into combinatorial properties. The important point here is, given a CIET $\delta$ satisfying the
Keane condition, the attracting subshift of the $\delta$-automorphism must contain pairs of points
representing the coding of the orbits of the $\delta$-singularities. This is the main idea
of the algorithm below.

Starting with a self-induced CIET $\delta$, we have constructed in the previous section the $\delta$-automorphism
using Rauzy inductions. The algorithm will attempt to execute this process in reverse by decomposing
a positive primitive automorphism into Dehn twists, using a combinatorial interpretation of the Rauzy induction.

If an automorphism $\varphi$ is a $\delta$-automorphism for some CIET $\delta$, then the $\varphi$-singularities
will be exactly the $\delta$-singularities, and in this case, we can perform (combinatorial) Rauzy inductions
using only the singularities. In the general case however, the set of singularities of a positive primitive
automorphism can be quite different from the set of singularities of a CIET, and
Rauzy inductions may not be possible. In section \ref{subsec:nececond}, we give a list
of conditions on the attracting subshift (and in particular the singularities) of a positive primitive
automorphism that are necessary for Rauzy inductions to be possible. The final algorithm is detailed
in section \ref{subsec:finalalg}; it consists in checking the necessary conditions of section \ref{subsec:nececond}
and applying combinatorial Rauzy inductions until the induction fails (in which case we conclude the automorphism
does not come from a CIET) or the automorphism is successfuly decomposed (in which case we can easily deduce the
underlying CIET).

	\subsection{The necessary conditions}\label{subsec:nececond}

	Let $\delta=(\pi, \lambda)$ be a self-induced CIET and let $\varphi$ be its $\delta$-automorphism. We explicit
	a list of conditions that are satisfied by both $\varphi$ and $\Sv$. As the algorithm
	attempts to decompose an automorphism into Dehn twists, we will use these conditions to determine
	if a Rauzy induction is possible; failing one of them will immediately stop the algorithm.
	Some of these conditions are deliberatly overly detailed to prepare for the algorithm.
	\begin{enumerate}[(C 1)]
		\item Following from proposition \ref{prop:ckeane}, there is exactly $2N-2$ distincts
		$\varphi$-singularities $\{\Omega_0, \dots, \Omega_{2N-3}\}$, each containing exactly $2$ points.
		\item From corollary \ref{cor:ckeane}, we can order the
		$\varphi$-singularities so that for any $0\le i\le N-2$,
		\begin{center}
			$\Omega_i = \{(U_{(i)}, V_{(i)}), (U_{(i)}', V_{(i)}')\}$ with $U_{(i)} = U_{(i)}'$
		\end{center}
		and for any $N-1\le j\le 2N-3$,
		\begin{center}
			$\Omega_j = \{(U_{(j)}, V_{(j)}), (U_{(j)}', V_{(j)}')\}$ with $V_{(j)} = V_{(j)}'$.
		\end{center}
		\item Define the \textbf{forward graph} $G_+$ (resp. \textbf{backward graph} $G_-$) as the graph whose nodes are the elements
		of $A_N$ and there is an (unoriented) edge from $a$ to $b$ if there exists $0\le i\le N-2$ (resp. $N-1\le j\le 2N-3$)
		such that $a$ and $b$ (resp. $a^{-1}$ and $b^{-1}$) are the first letters of $V_{(i)}$ and $V_{(i)}'$ (resp. $U_{(j)}$ and $U_{(j)}'$);
		such an edge is labeled by $a_0$ if $a_0^{-1}$ (resp. $a_0$) is the first letter of $U_{(i)}$ (resp. $V_{(j)}$)
		(see the example in appendix \ref{appendix:example}). Define the \textbf{distance} between
		two nodes (resp. a node and an edge) as the number of edges contained in the path joining
		one to the other (note: conventionnally, the path joining a node $a$ to an edge $e$ does not contain $e$).
		We get the following conditions from theorem \ref{thm:othermainresult}.
		\begin{enumerate}[(C 3.1)]
			\item Both graphs are connected and each one contains two nodes of degree
			(the number of adjacent edges) $1$ while all the others have degree $2$.
			\item Observe that $\pi_0^{-1}(0) = \alpha$ (resp. $\pi_1^{-1}(0) = \beta$) is a node of $G_+$ (resp. $G_-$)
			with degree $1$. Moreover, for any node $a$ of $G_+$ (resp. $G_-$), $\pi_0(a)$ (resp. $\pi_1(a)$)
			is given by the distance between $a$ and $\alpha$ (resp. $a$ and $\beta$).
			\item Let $e_a$ and $e_b$ be two edges of $G_+$ (resp. $G_-$) labeled $a$ and $b$ respectively and suppose $a\ne b$.
			Then $\pi_1(a) < \pi_1(b)$ (resp. $\pi_0(a) < \pi_0(b)$) if and only if $e_a$ is closer to $\alpha$ (resp. $\beta$)
			than $e_b$.

			It may happen that all the edges of $G_+$ have a common label $\beta_0$ and all the edges of $G_-$ have
			a common label $\beta_1$. In that case, we prove the following proposition.
			\begin{prop}\label{prop:determinepi}
			We have $\beta_0 = \beta_1$ and $\{\pi_0(\beta_0), \pi_1(\beta_0)\} = \{0, N-1\}$.
			\end{prop}
			\begin{proof}
			Since all the forward $\delta$-singularities
			are contained in the domain $[y_{\beta_0}, y_{\beta_0}']$ of $\delta_{\beta_0^{-1}}$, there exist $a, b\in A_N$ with
			$a\ne b$ such that $y_{\beta_0}$ (resp. $y_{\beta_0}'$) is in the domain of
			$\delta_{a}$ (resp. $\delta_{b}$). If $\pi_1(\beta_0)$ is neither $0$ nor $N-1$,
			then both $y_{\beta_0}$ and $y_{\beta_0}'$ are backward $\delta$-singularities and we have a contradiction.
			Obviously, the same reasoning tells us $\pi_0(\beta_1)$ must also be either $0$ or $N-1$.
			We deduce $\lambda_{\beta_0} > \sum\limits_{a\ne \beta_0} \lambda_a$ and
			$\lambda_{\beta_1} > \sum\limits_{a\ne \beta_1} \lambda_a$ and conclude $\beta_0=\beta_1$.
			Finally, we observe the equality $\pi_0(\beta_0)=\pi_1(\beta_0)$ contradicts the Keane condition.
			\end{proof}
		\end{enumerate}
		\item Define $\alpha_0 = \pi_0^{-1}(N-1)$ and $\alpha_1 = \pi_1^{-1}(N-1)$.
		There is exactly one point $(U, V)$ of one $\varphi$-singularity $\Omega$
		such that $U_0^{-1} = \alpha_1$ and $V_0 = \alpha_0$. The $\varphi$-singularity $\Omega$
		contains a point $(U', V)$ with $U'\ne U$ if and only if $\delta$ has type $0$.
		\item If $\Sd(0) = \{Z\}$, then $\partial^2 \varphi(Z)=Z$.
	\end{enumerate}

	We conjecture that condition (C $1$) is a (necessary and) sufficient condition for the
	attracting subshift of an $A_N$-positive primitive automorphism to be the subshift of a CIET.
	\begin{con}
	If $\psi$ is an $A_N$-positive primitive automorphism satisfying condition (C $1$),
	then there exists a self-induced CIET $\delta$ such that $\Sd=\Spsi$.
	In particular, $\psi$ also satisfies (C $2$) and (C $3$) and if $\varphi$ is the
	$\delta$-automorphism, then $i_u\circ \psi^k=\varphi^h$ for some $h,k\ge 1$ and $u\in F_N$.
	\end{con}

	\subsection{The algorithm}\label{subsec:finalalg}

	The algorithm is based on the ability to identify the singularities of an $A_N$-positive primitive
	automorphism. The reader is referred to \cite{Jul} for a complete approach of the problem;
	the relevant results of \cite{Jul} are summarized in appendix \ref{appendix:thebase}.

	The algorithm does not pretend to be optimal. It may surely be improved in many ways, as there are many more necessary
	conditions than the ones we have listed. In fact, the aim here is to present a minimal list of such conditions. Another
	good point is that the approach we use is fully combinatorial.

	Let $\psi$ be an $A_N$-positive primitive automorphism. The following algorithm will
	determine in a finite time if its attracting subshift $\Spsi$ is equal to $\Sd$ for some
	CIET $\delta$. An example is detailed in appendix \ref{appendix:example}.
	\begin{enumerate}[(1)]
		\item List the $\psi$-singularities using the algorithm of \cite{Jul}. Stop if conditions
		(C $1$) and (C $2$) are not both satisfied.
		\item Define the forward graph $G_+$ and the backward graph $G_-$ as in condition (C $3$), and stop
		if condition (C $3.1$) is not satisfied. We now define a pair $\pi = (\pi_0, \pi_1)$
		with $\pi_0,\pi_1:A_N\to \{0, \dots, N-1\}$ that will agree with condition (C $3.2$).
		Choose a node $\alpha$ with degree $1$ in $G_+$. For any $a\in A_N$, define $\pi_0(a)$
		as the distance between $a$ and $\alpha$. If $G_+$ contains two edges with distinct labels,
		apply $(2.1)$. If all the edges of $G_+$ have the same label and $G_-$ contains two edges
		with distinct labels, apply $(2.2)$. Apply $(2.3)$ otherwise.
		\begin{enumerate}[(2.1)]
			\item Choose two edges $e_a$ and $e_b$ of $G_+$ labeled $a$ and $b$
			with $a\ne b$ and such that $e_a$ is closer to $\alpha$ than $e_b$.
			Choose the node $\beta$ of degree $1$ of $G_-$
			such that $a$ is closer to $\beta$ than $b$.
			For any $a_0\in A_N$, define $\pi_1(a_0)$ as the distance between the nodes
			$a_0$ and $\beta$.
			\item Choose two edges $e_a$ and $e_b$ of $G_-$ labeled $a$ and $b$
			with $a\ne b$ and such that $a$ is closer to $\alpha$ than $b$ in $G_+$.
			Choose the node $\beta$ of degree $1$ of $G_-$ such that $e_a$ is closer to $\beta$ than $e_b$.
			For any $a_0\in A_N$, define $\pi_1(a_0)$ as the distance between the nodes
			$a_0$ and $\beta$.
			\item Suppose all the edges of $G_+$ have the same label $\beta_0$ and all the edges
			of $G_-$ have the same label $\beta_1$. Stop if $\beta_0\ne \beta_1$.
			Also stop if $\alpha$ does not have degree $1$ in $G_-$.
			For any $a_0\in A_N$ define $\pi_1(a_0) = N-1-k$ where $k$
			is the distance between $a_0$ and $\alpha$ (in $G_-$).
		\end{enumerate}
		In any case, stop if conditions (C $3.2$), (C $3.3$) and (C $4$) are not all satisfied. We also
		check if $\pi$ is reducible and stop if it is.
		\item We now deal with condition (C $5$). Choose the smallest positive integer $k$ such that
		any $\psi$-singularity is fixed by $i_w\circ \psi^k$ for some $w\in F_N$.
		Let $\Omega = \{(U', V), (U, V)\}$ be the backward $\psi$-singularity
		such that $U_0 = (\pi_0^{-1}(0))^{-1}$ and $\pi_1(U_0'^{-1}) = \pi_1(U_0^{-1})-1$.
		Define $Z = U_0^{-1}(U, V)$ and $u\in F_N$ such that $\partial^2 (i_u\circ \psi^k)(Z)=Z$; also
		define $\varphi' = i_u\circ \psi^k$.

		Following from the discussion in section \ref{subsubsec:mirror}, we may also want to consider
		the pair $\pi'$ defined for any $a\in A_N$ by $\pi_0'(a) = N-1-\pi_0(a)$ and $\pi_1'(a) = N-1-\pi_1(a)$.
		Let $\Omega' = \{(X', Y), (X, Y)\}$ be the backward $\psi$-singularity
		such that $X_0 = (\pi_0'^{-1}(0))^{-1}$ and $\pi_1'(X_0'^{-1}) = \pi_1'(X_0^{-1})-1$.
		Define $Z' = X_0^{-1}(X, Y)$ and $v\in F_N$ such that $\partial^2 (i_v\circ \psi^k)(Z')=Z'$; also
		define $\varphi'' = i_v\circ \psi^k$.

		According to proposition \ref{prop:powerofvarphi}, we can stop if neither $\varphi'$ nor $\varphi''$ is $A_N$-positive.
		If $\varphi'$ (resp. $\varphi''$) is $A_N$-positive, we define $\varphi = \varphi'$ (resp. $\varphi = \varphi''$).
		Thanks to proposition \ref{prop:mirror}, we may define $\varphi$ to be either one of them if they are both $A_N$-positive.

		\item Set up a counter $i\leftarrow 0$. Define $\{\Omega_j^{(i)}, 0\le j\le 2N-3\}$ as the set of $\psi$-singularities,
		$(\pi_0^{(i)}, \pi_1^{(i)}) = (\pi_0, \pi_1)$ and $G_+^{(i)} = G_+$.
		\item Define $\alpha_0^{(i)} = (\pi_0^{(i)})^{-1}(N-1)$ and $\alpha_1^{(i)} = (\pi_1^{(i)})^{-1}(N-1)$ and let $e$ be the (only) edge
		of $G_+^{(i)}$ adjacent to $\alpha_0^{(i)}$.
		\begin{enumerate}[(5.1)]
			\item If $e$ is not labeled $\alpha_1^{(i)}$ define the automorphism $\sigma_i$ such that $\sigma_i(\alpha_1^{(i)}) = \alpha_1^{(i)}\alpha_0^{(i)}$
			and $\sigma_i(a)=a$ for any $a\ne \alpha_1^{(i)}$.
			\item If $e$ is labeled $\alpha_1^{(i)}$, define the automorphism
			$\sigma_i$ such that $\sigma_i(\alpha_0^{(i)})=\alpha_1^{(i)}\alpha_0^{(i)}$ and $\sigma_i(a)=a$ for any $a\ne \alpha_0^{(i)}$.
		\end{enumerate}
		\item
		\begin{itemize}
			\item If the automorphism $\sigma_i^{-1}\circ \dots\circ \sigma_0^{-1}\circ \varphi$ is the identity, the
			algorithm is a success. Define the incidence matrix $M$ of $\psi^k$ as the matrix $N\times N$ where
			$M_{a,b}$ is defined for any $a, b\in A_N$ as the number of occurences of the letter $a$ in $\psi^k(b)$.
			The matrix $M$ is primitive (one of its power only has positive entries) since $\psi^k$ is $A_N$-primitive.
			Perron-Frobenius theorem tells us the eigenvalue with greatest modulus is simple, has modulus (strictly)
			greater than the modulus of the other eigenvalues, is a positive real number and has a positive eigenvector
			$\lambda$. Define $\delta = (\pi, \lambda)$ where $\pi$ is as above.
			We have $\Spsi = \Sd$. The smallest period of the sequence $(\sigma_j)_{0\le j\le i}$ will give the $\delta$-automorphism.
			\item If the automorphism $\sigma_i^{-1}\circ \dots\circ \sigma_0^{-1}\circ \varphi$ is not the identity but is $A_N$-positive,
			we continue on to step $(7)$. Observe any letter of $A_N$ appears in $\sigma_{i-1}^{-1}\circ \dots\circ \sigma_0^{-1}\circ \varphi(a)$
			for at least one letter $a\in A_N$ (otherwise $\sigma_{i-1}^{-1}\circ \dots\circ \sigma_0^{-1}\circ \varphi$ would not be an automorphism)
			and deduce that the automorphism $\sigma_i^{-1}\circ \dots\circ \sigma_0^{-1}\circ \varphi$
			is (strictly) shorter (with respect to the lengths of the images of elements of $A_N$) than
			$\sigma_{i-1}^{-1}\circ \dots\circ \sigma_0^{-1}\circ \varphi$, ensuring the algorithm will effectively
			end after a finite, easily bounded, number of steps.
			\item According to proposition \ref{prop:powerofvarphi}, we can stop if the automorphism
			$\sigma_i^{-1}\circ \dots\circ \sigma_0^{-1}\circ \varphi$ is not $A_N$-positive.
		\end{itemize}
		\item We define a new set of singularities. We assume $\Omega_k^{(i)}$ contains the (only) point
		$(U, V)$ with $U_0^{-1} = \alpha_1^{(i)}$ and $V_0 = \alpha_0^{(i)}$. For any $1\le j\le 2N-3$, $j\ne k$,
		we simply define $\Omega_j^{(i+1)} = \partial^2\sigma_i^{-1} (\Omega_j^{(i)})$.
		Observe that if $\sigma_i$ is defined as in step $(5.1)$ (resp. $(5.2)$), then $\Omega_k^{(i)}$ contains
		a point $(U', V)$ with $U'\ne U$ (resp. $(U, V')$ with $V\ne V'$) and apply step $(7.1)$ (resp. $(7.2)$).
		\begin{enumerate}[(7.1)]
			\item Define $\Omega_k^{(i+1)} = \partial^2\sigma_i^{-1} (S (\Omega_k^{(i)}))$. Define $(\pi_0^{(i+1)}, \pi_1^{(i+1)})$
			as the result of a type $0$ induction on $(\pi_0^{(i)}, \pi_1^{(i)})$. Namely, $\pi_0^{(i+1)} = \pi_0^{(i)}$ and
			\begin{itemize}
				\item $\forall a\in A_N;~\pi_1^{(i)}(a)\le \pi_1^{(i)}(\alpha_0^{(i)}),~\pi_1^{(i+1)}(a) = \pi_1^{(i)}(a)$,
				\item $\pi_1^{(i+1)}(\alpha_1^{(i)}) = \pi_1^{(i)}(\alpha_0^{(i)})+1$,
				\item $\forall a\in A_N;~\pi_1^{(i)}(\alpha_0^{(i)}) < \pi_1^{(i)}(a) < \pi_1^{(i)}(\alpha_1^{(i)}),~\pi_1^{(i+1)}(a)=\pi_1^{(i)}(a)+1$.
			\end{itemize}
			\item Define $\Omega_k^{(i+1)} = \partial^2\sigma_i^{-1} (S^{-1} (\Omega_k^{(i)}))$. Define $(\pi_0^{(i+1)}, \pi_1^{(i+1)})$
			as the result of a type $1$ induction on $(\pi_0^{(i)}, \pi_1^{(i)})$. Namely, $\pi_1^{(i+1)} = \pi_1^{(i)}$ and
			\begin{itemize}
				\item $\forall a\in A_N;~\pi_0^{(i)}(a)\le \pi_0^{(i)}(\alpha_1^{(i)}),~\pi_0^{(i+1)}(a) = \pi_0^{(i)}(a)$,
				\item $\pi_0^{(i+1)}(\alpha_0^{(i)}) = \pi_0^{(i)}(\alpha_1^{(i)})+1$,
				\item $\forall a\in A_N;~\pi_0^{(i)}(\alpha_1^{(i)}) < \pi_0^{(i)}(a) < \pi_0^{(i)}(\alpha_0^{(i)}),~\pi_0^{(i+1)}(a)=\pi_0^{(i)}(a)+1$.
			\end{itemize}
		\end{enumerate}
		Define the new graphs $G_+^{(i+1)}$ and $G_-^{(i+1)}$ and note that conditions (C $1$), (C $2$), (C $3.1$) and (C $3.2$) are automatically satisfied.
		Stop if conditions (C $3.3$) and (C $4$) are not both satisfied. Set $i\leftarrow i+1$ and go back to step $(5)$.
	\end{enumerate}

\appendix
\section{Identifying the singularities of an $A_N$-positive primitive automorphism}\label{appendix:thebase}

The article \cite{Jul} gives an algorithm to find the singularities of an $A_N$-positive primitive automorphism.
The main tool of the algorithm is a careful study of the prefix-suffix representation of an attracting subshift.
The reader is referred fo \cite{CanSie} for results on the prefix-suffix representation. We briefly recall here the
slightly adapted version of \cite{Jul}.

Let $\varphi$ be an $A_N$-positive primitive automorphism and $\Sv$ its attracting subshift.
The \textbf{prefix-suffix automaton} of $\varphi$ is defined as follows:
\begin{itemize}
	\item $A_N$ is its set of vertices,
	\item $P = \{(p, a, s)\in A_N^*\times A_N\times A_N^*; \exists b\in A_N; \varphi(b) = pas\}$ (where $A_N^*$
	is the set of words of $F_N$ with letters in $A_N$ and $\epsilon\in A_N^*$) is the set of labels,
	\item there is an edge labeled $(p, a, s)$ from $a$ to $b$ if and only if $\varphi(b) = pas$.
\end{itemize}
An example is given on figure \ref{fig:apsex}.
\begin{figure}[h!]
\begin{center}
	\scalebox{0.5}{\includegraphics{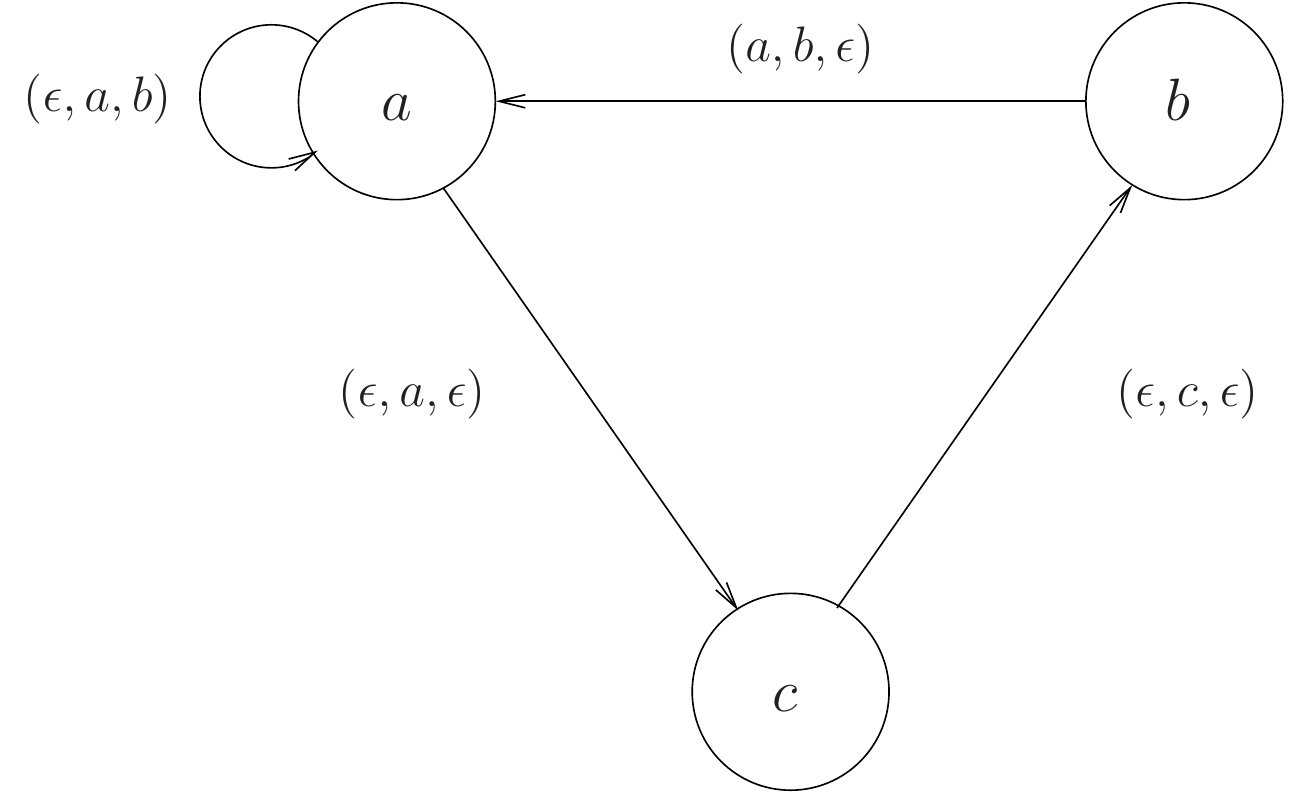}}
\end{center}
\caption{Prefix-suffix automaton associated to $\phi : a\mapsto ab, b\mapsto c, c\mapsto a$.}
\label{fig:apsex}
\end{figure}

The set of sequences of labels of infinite walks in this automaton is denoted $D$; it is the set of admissible developments.
\begin{prop}\label{prop:ref}
If $(p_i, a_i, s_i)_{i\ge 0}$ is in $D$, then for all $n\in \mathds{N}$, $\varphi(a_{n+1}) = p_na_ns_n$.
\end{prop}
In \cite{CanSie}, the authors define a map
$\rho_{\varphi}:\Sv\to D$ which gives a representation of the attracting subshift's structure.
If $W = (U, V)$ is a point of $\Sv$, then $\rho_{\varphi}(W) = (p_i, a_i, s_i)_{i\ge 0}$ is called
the \textbf{prefix-suffix development} of $W$, and it is such that
\begin{itemize}
	\item if $(s_i)_{i\in \mathds{N}}$ is not eventually constant equal to $\epsilon$,
	then $V = \lim\limits_{n\to +\infty} a_0s_0\varphi(s_1)\dots \varphi^n(s_n)$,
	\item if $(p_i)_{i\in \mathds{N}}$ is not eventually constant equal to $\epsilon$,
	then $U = \lim\limits_{n\to +\infty} p_0^{-1}\varphi(p_1^{-1})\dots \varphi^n(p_n^{-1})$.
\end{itemize}
Those developments whose prefix or suffix sequence end up being constant equal to $\epsilon$ are identified in \cite{CanSie}.
\begin{thm}[\cite{CanSie}]
The map $\rho_{\varphi}$ is continuous and onto. Any development $d\in D$ has at most $N$ pre-images.
\end{thm}

We now recall the relevant results of \cite{Jul}. We still consider $\varphi$ is an $A_N$-positive primitive automorphism
and $\Sv$ is its attracting subshift. Observe that for any positive integer $k$, the automorphism $\varphi^k$
is also $A_N$-positive and primitive and its attracting subshift $\Sigma_{\varphi^k}$ verifies
$\Sigma_{\varphi^k} = \Sv$. For any $k\ge 1$,
let $\gmk$ and $\gpk$ be the $F_N$ to $F_N$ maps defined, for every $u=u_0u_1\dots u_p$ in $F_N$, by
\begin{itemize}
	\item $\gmk(u) = \varphi^k(u_p)u_0u_1\dots u_{p-1}$,
	\item $\gpk(u) = u_1\dots u_{p-1}u_p\varphi^k(u_0)$.
\end{itemize}
The search for $\varphi$-singularities is based on the following result.
\begin{thm}[\cite{Jul}]\label{thm:mainreference}
Let $W$ and $W'$ be two distinct points of $\Sv$ with $\rho_{\varphi^k}(W) = (p, a, s)*$
and $\rho_{\varphi^k}(W') = (q, b, r)*$ (where $\rho_{\varphi^k}$ is the prefix-suffix development
map of $\varphi^k$ and the symbol $*$ indicates the triplet is repeated indefinitely).
\begin{itemize}
	\item If for any $i,j\in \mathds{N}$, we have $\gmk^i(p)\ne \gmk^j(q)$ (resp. $\gpk^i(s)\ne \gpk^j(r)$),
	then for any $i,j\in \mathds{N}$, the points $S^{-i}(W)$ and $S^{-j}(W')$ (resp. $S^{i+1}(W)$ and $S^{j+1}(W')$)
	do not belong to a common singularity.
	\item If $i$ and $j$ are the smallest integers such that $\gmk^i(p) = \gmk^j(q)$ (resp. $\gpk^{i}(s) = \gpk^{j}(r)$),
	then $S^{-i}(W)$ and $S^{-j}(W')$ (resp. $S^{i+1}(W)$ and $S^{j+1}(W')$) belong to the same singularity $\Omega$.
	Moreover, the singularity $\Omega$ is fixed by $(i_w\circ \varphi^k)^h$
	for some integer $h\ge 1$ and $w = \gmk^i(p)$ (resp. $w^{-1} = \gpk^{i}(s)$).
\end{itemize}
Conversely, if $W_{(0)}$ is a point of $\Sv$ belonging to a singularity $\Omega$, then
there exists $W_{(1)}\in \Omega$, an integer $k\le 4N-4$ and two points $W$ and $W'$ of $\Sv$ such that
\begin{itemize}
	\item $\rho_{\varphi^k}(W) = (p, a, s)*$, $\rho_{\varphi^k}(W') = (q, b, r)*$,
	\item $\gmk^i(p) = \gmk^j(q)$ (resp. $\gpk^i(s) = \gpk^j(r)$) for some integers $i,j\ge 0$,
	\item $W_{(0)} = S^{-i}(W)$ and $W_{(1)} = S^{-j}(W')$ (resp. $W_{(0)} = S^{i+1}(W)$ and $W_{(1)} = S^{j+1}(W')$).
\end{itemize}
\end{thm}
Keeping the notation of the theorem above, note that it is possible for $S^{-i}(W)$ and $S^{-j}(W')$
(resp. $S^{i+1}(W)$ and $S^{j+1}(W')$) to
belong to the same singularity even if $i$ and $j$ are not the smallest integers for which
$\gmk^i(p) = \gmk^j(q)$ (resp. $\gpk^{i}(s) = \gpk^{j}(r)$).
This is typical of singularities containing two distinct points $(U, V)$ and $(U', V')$ with both $U_0 = U_0'$ and $V_0=V_0'$
(an example is given in \cite{Jul}).

Recall that different points may have the same prefix-suffix development.
As theorem \ref{thm:mainreference} only take the prefix-suffix development
into consideration, if using theorem \ref{thm:mainreference} tells us $S^m(W)$ and $S^n(W')$ belong to $\Omega$, then
we also know that for any point $W_{(0)}$
(resp. $W_{(0)}'$) such that $\rho_{\varphi^k}(W) = \rho_{\varphi^k}(W_{(0)})$ (resp. $\rho_{\varphi^k}(W') = \rho_{\varphi^k}(W_{(0)}')$),
the point $S^m(W_{(0)})$ (resp. $S^n(W_{(0)}')$) belongs to $\Omega$.\\

The algorithm to find $\varphi$-singularities simply consists, for each $1\le k\le 4N-4$,
in running pairs $((p, a, s)*, (q, b, r)*)$ of constant (with respect to $\rho_{\varphi^k}$) prefix-suffix developments
through theorem \ref{thm:mainreference}. It is obvious from proposition \ref{prop:ref}
that there is a finite number of such developments. Moreover, it is explained in \cite{Jul}
how properties of $\varphi^{-k}$ can be used to bound the minimal integers $i,j$ such that
$\gmk^i(p) = \gmk^j(q)$ (resp. $\gpk^{i}(s) = \gpk^{j}(r)$), ensuring the algorithm
will end after a finite number of steps.

Note that it may not be necessary to sweep through all constant prefix-suffix developments pairs,
as the overall number of points contained in singularities is bounded (see \cite{GJLL}, \cite{Jul}).
This bound is reached for CIETs.

\begin{rmk} One may be interested in finding the $\varphi$-singularities when
$\varphi$ is the $\delta$-automorphism of some CIET $\delta$. From proposition \ref{prop:invsing},
we only need to study developments that are constant for $\rho_\varphi$.
\end{rmk}

\section{An example}\label{appendix:example}

Let $\psi$ be the $\{a, b, c, d\}$-positive primitive automorphism defined by
\begin{center}
	\begin{tabular}{ccccl}
		$\psi$ & $:$ & $a$ & $\mapsto$ & $bdacda$\\
		&& $b$ & $\mapsto$ & $bdbda$\\
		&& $c$ & $\mapsto$ & $ccda$\\
		&& $d$ & $\mapsto$ & $cda$
	\end{tabular}
\end{center}
There are a lot of obvious pairs yielding $\psi$-singularities. We get:
\begin{itemize}
	\item $\Omega_0 = \{W_{(0)}, W_{(0)}'\}$ with $\rho_\psi(W_{(0)}) = (\epsilon, b, dbda)*$ and $\rho_\psi(W_{(0)}') = (\epsilon, c, cda)*$,
	\item $\Omega_1 = \{W_{(1)}, W_{(1)}'\}$ with $\rho_\psi(W_{(1)}) = (bd, a, cda)*$ and $\rho_\psi(W_{(1)}') = (bd, b, da)*$,
	\item $\Omega_2 = \{W_{(2)}, W_{(2)}'\}$ with $\rho_\psi(W_{(2)}) = (c, c, da)*$ and $\rho_\psi(W_{(2)}') = (c, d, a)*$,
	\item $\Omega_3 = \{W_{(3)}, W_{(3)}'\}$ with $\rho_\psi(S^{-1}(W_{(3)})) = (bd, a, cda)*$ and $\rho_\psi(S^{-1}(W_{(3)}')) = (\epsilon, c, cda)*$,
	\item $\Omega_4 = \{W_{(4)}, W_{(4)}'\}$ with $\rho_\psi(S^{-1}(W_{(4)})) = (bd, b, da)*$ and $\rho_\psi(S^{-1}(W_{(4)}')) = (c, c, da)*$.
\end{itemize}
Also, observe $\gamma_{\psi_+}(a) = \gamma_{\psi_+}(dbda)$, and deduce there is a singularity $\Omega_5 = \{W_{(5)}, W_{(5)}'\}$ with
$\rho_\psi(S^{-2}(W_{(5)})) = (c, d, a)*$ and $\rho_\psi(S^{-2}(W_{(5)}')) = (\epsilon, b, dbda)*$.

\begin{rmk}
Define $U = \lim\limits_{n\to +\infty} \psi^n(a^{-1})$. The point $U$ is the first coordinate of the points $W_{(0)}$, $W_{(0)}'$, $S^{-1}(W_{(3)}')$ and $S^{-2}(W_{(5)}')$.
All the other coordinates are explicitly given by the prefix-suffix developments.
\end{rmk}

The forward and backward graphs are given on figures \ref{fig:forward1} and \ref{fig:backward1}.
\begin{figure}[h!]
\begin{center}
	\scalebox{0.5}{\includegraphics{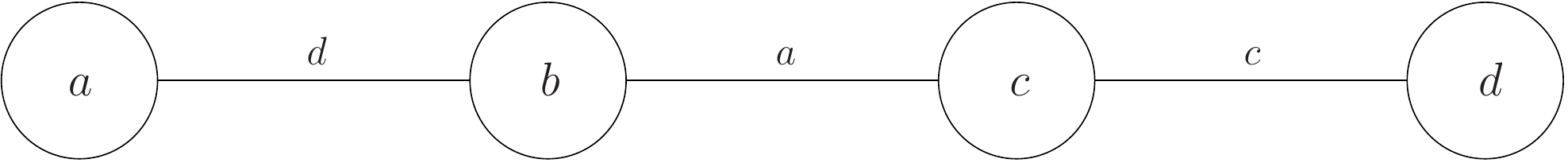}}
\end{center}
\caption{First forward graph.}
\label{fig:forward1}
\end{figure}
\begin{figure}[h!]
\begin{center}
	\scalebox{0.5}{\includegraphics{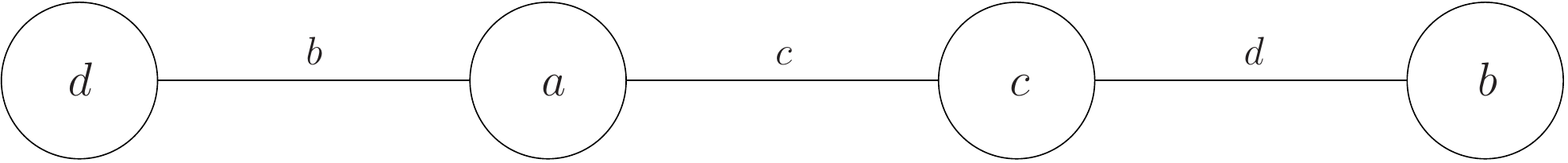}}
\end{center}
\caption{First backward graph.}
\label{fig:backward1}
\end{figure}

Using step $(2.1)$, we can define the pair $\pi = (\pi_0, \pi_1)$:
\begin{center}
	\begin{tabular}{ccccccccc}
		$\pi_0:$ & $a$ & $\mapsto$ & $0$ & \hspace{2cm} & $\pi_1:$ & $a$ & $\mapsto$ & $1$\\
		& $b$ & $\mapsto$ & $1$ &&& $b$ & $\mapsto$ & $3$\\
		& $c$ & $\mapsto$ & $2$ &&& $c$ & $\mapsto$ & $2$\\
		& $d$ & $\mapsto$ & $3$ &&& $d$ & $\mapsto$ & $0$
	\end{tabular}
\end{center}
Let $(U, V)$ be the point of $\Omega_5$ such that $U_0 = a^{-1}$; define $Z = a(U, V)$. The singularity $\Omega_5$ is fixed by $i_{(bdacda)^{-1}}\circ \psi$,
and we deduce $Z$ is fixed by $\partial^2 (i_{a^{-1}}\circ i_{(bdacda)^{-1}}\circ \psi\circ i_a)$. Define
$\varphi = i_{a^{-1}}\circ i_{(bdacda)^{-1}}\circ \psi\circ i_a = i_{\psi(a)(abdacda)^{-1}}\circ \psi = i_{a^{-1}}\circ \psi$. We obtain the
$\{a, b, c, d\}$-positive automorphism
\begin{center}
	\begin{tabular}{ccccl}
		$\varphi$ & $:$ & $a$ & $\mapsto$ & $abdacd$\\
		&& $b$ & $\mapsto$ & $abdbd$\\
		&& $c$ & $\mapsto$ & $accd$\\
		&& $d$ & $\mapsto$ & $acd$
	\end{tabular}
\end{center}
Note that, in this case, working with $\pi' = (\pi_0', \pi_1')$ defined by $\pi_0' = 3-\pi_0(a_0)$ and $\pi_1' = 3-\pi_1(a_0)$
for any $a_0\in \{a, b, c, d\}$ would not have provided us with an $\{a, b, c, d\}$-positive automorphism.

From step $(5)$, we define $\sigma_0(b)=bd$ and $\sigma_0(a_0)=a_0$ if $a_0\ne b$.
We obtain
\begin{center}
	\begin{tabular}{ccccl}
		$\sigma_0^{-1}\circ \varphi$ & $:$ & $a$ & $\mapsto$ & $abacd$\\
		&& $b$ & $\mapsto$ & $abb$\\
		&& $c$ & $\mapsto$ & $accd$\\
		&& $d$ & $\mapsto$ & $acd$
	\end{tabular}
\end{center}
which is still $A_N$-positive, and continue on to step $(7)$. Obviously, we only need the first few letters
of the coordinates of the points contained in the singularities to move on. Following from step $(7)$
(in this case $(7.1)$), we obtain the new graphs of figures \ref{fig:forward2} and \ref{fig:backward2}.
\begin{figure}[h!]
\begin{center}
	\scalebox{0.5}{\includegraphics{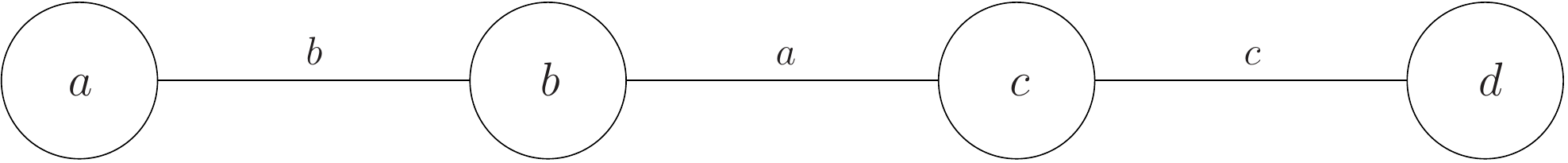}}
\end{center}
\caption{Second forward graph.}
\label{fig:forward2}
\end{figure}
\begin{figure}[h!]
\begin{center}
	\scalebox{0.5}{\includegraphics{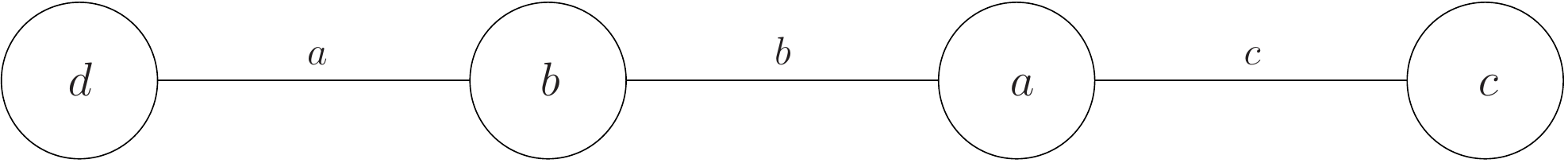}}
\end{center}
\caption{Second backward graph.}
\label{fig:backward2}
\end{figure}

This allows us (step $(5)$) to define $\sigma_1(d)=cd$ and $\sigma_1(a_0)=a_0$ if $a_0\ne d$,
and we obtain
\begin{center}
	\begin{tabular}{ccccl}
		$\sigma_1^{-1}\circ \sigma_0^{-1}\circ \varphi$ & $:$ & $a$ & $\mapsto$ & $abad$\\
		&& $b$ & $\mapsto$ & $abb$\\
		&& $c$ & $\mapsto$ & $acd$\\
		&& $d$ & $\mapsto$ & $ad$
	\end{tabular}
\end{center}
which is again $A_N$-positive.

The algorithm cycles another six times. One may check that we obtain
\begin{center}
	\begin{tabular}{clcclccl}
		$\sigma_2:$ & $a\mapsto a$ & \hspace{6mm} & $\sigma_3:$ & $a\mapsto a$ & \hspace{6mm} & $\sigma_4:$ & $a\mapsto a$\\
		& $b\mapsto b$ & & & $b\mapsto b$ & & & $b\mapsto b$\\
		& $c\mapsto cd$ & & & $c\mapsto c$ & & & $c\mapsto ac$\\
		& $d\mapsto d$ & & & $d\mapsto ad$ & & & $d\mapsto d$\\ \\
		$\sigma_5:$ & $a\mapsto ab$ & & $\sigma_6:$ & $a\mapsto a$ & & $\sigma_7:$ & $a\mapsto ad$\\
		& $b\mapsto b$ & & & $b\mapsto ab$ & & & $b\mapsto b$\\
		& $c\mapsto c$ & & & $c\mapsto c$ & & & $c\mapsto c$\\
		& $d\mapsto d$ & & & $d\mapsto d$ & & & $d\mapsto d$
	\end{tabular}
\end{center}
and $\varphi = \sigma_0\circ \sigma_1\circ \sigma_2\circ \sigma_3\circ \sigma_4\circ \sigma_5\circ \sigma_6\circ \sigma_7$.
Define the incidence matrix
\begin{center}
	$M = \begin{bmatrix}2&1&1&1\\1&2&0&0\\1&0&2&1\\2&2&1&1\end{bmatrix}$
\end{center}
of $\psi$, its dominant eigenvalue $\eta$ and choose a positive eigenvector $\lambda$ associated to $\eta$.
The CIET $\delta = (\pi, \lambda)$ satisfies $\Sd = \Spsi$, and $\varphi$ is the $\delta$-automorphism.

\bibliography{bibli}{}
\bibliographystyle{alpha}

\end{document}